\documentclass[a4paper,10pt]{article}

\usepackage{amsmath, amssymb, amsthm, mathtools}
\usepackage{cite}
\usepackage{color,soul}  
\usepackage{flafter}  
\usepackage{fullpage}
\usepackage{graphicx}
\usepackage[mathscr]{eucal}
\usepackage[small]{caption}
\usepackage{subfigure}
\usepackage{booktabs}
\usepackage{epstopdf}
\usepackage{bm} 

\newif\ifupdatetikz
\updatetikzfalse 
\ifupdatetikz
	\usepackage{tikz}
	\usepackage{pgfplots}
	\pgfplotsset{
		compat=newest,
		tick label style={font=\scriptsize},
		label style={font=\scriptsize},
		legend style={font=\tiny}
	}
	\usepgfplotslibrary{external}
	\usetikzlibrary{shapes,arrows,positioning}
	\tikzset{cross/.style={cross out, draw=black, minimum size=2*(#1-\pgflinewidth), inner sep=0pt, outer sep=0pt},cross/.default={2pt}}
	\tikzstyle{axis} = [-latex',line width=1.25]
	
\else
	\usepackage{tikzexternal}
\fi
\newcommand{\figname}[1]{\tikzsetnextfilename{#1}}

\newtheoremstyle{mystyle}
  {}
  {}
  {\itshape}
  {}
  {\bfseries}
  {.}
  { }
  {\thmname{#1}\thmnumber{ #2}\thmnote{ (#3)}}

\theoremstyle{mystyle}

\newtheorem{thm}{Theorem}[section]

\newtheorem{ex}[thm]{Example}

\newtheorem{rem}[thm]{Remark}

\let\originalleft\left
\let\originalright\right
\renewcommand{\left}{\mathopen{}\mathclose\bgroup\originalleft}
\renewcommand{\right}{\aftergroup\egroup\originalright}

\newcommand{\abs}[1]{\lvert#1\rvert}

\newcommand{\averageV}[1]{\left\langle #1 \right\rangle}
\newcommand{\Bcal}{\boldsymbol{\mathcal{B}}}

\newcommand{\C}{\mathbb{C}}			

\newcommand{\cell}{\mathcal{C}}		

\newcommand{\DF}{\mathbf{D}}
\newcommand{\Dt}{\Delta t}
\newcommand{\Dtime}{\mathrm{D}_t}
\newcommand{\Dxv}{\mathrm{D}_{\xGrid,\vGrid}}
\newcommand{\diag}[1]{\operatorname{diag}\left(#1\right)}

\newcommand{\diffMatrix}{\mathbf{B}}
\newcommand{\discriminant}{\Delta}
\newcommand{\disk}{\mathcal{D}}
\newcommand{\diskpfe}{\disk^{\text{\scriptsize{PFE}}}}
\newcommand{\diskprk}[1]{\disk^{\text{\scriptsize{PRK},#1}}}
\newcommand{\dt}{\delta t}
\newcommand{\dx}{\Delta x}
\newcommand{\dy}{\Delta y}
\newcommand{\epsi}{\varepsilon}

\newcommand{\eye}{\mathbf{I}}

\newcommand{\F}{\mathbf{F}}					

\newcommand{\fExact}{\tilde{\fSDSys}} 		
\newcommand{\fF}{\hat{\mathbf{F}}}
\newcommand{\fSDSys}{\mathbf{f}} 		

\newcommand{\hydroLimit}{\epsi \to 0}
\newcommand{\lambdavec}{\boldsymbol{\lambda}}

\newcommand{\Ma}{\mathcal{M}}						

\newcommand{\MF}{\mathbf{M}}

\newcommand{\norm}[1]{\lVert#1\rVert}

\newcommand{\PF}{\mathbf{P}}
\newcommand{\R}{\mathbb{R}}				

\newcommand{\Sdt}{S_{\dt}}
\newcommand{\SdtF}{\mathbf{S}_{\dt}}

\newcommand{\set}[2]{\left(#1\right)_{#2=1}^{\uppercase{#2}}}

\newcommand{\Sp}[1]{\mathrm{Sp}\left(#1\right)}

\newcommand{\tauvec}{\boldsymbol{\tau}}
\newcommand{\uSys}{\mathbf{u}}			
\newcommand{\uSDSys}{\boldsymbol{u}} 	
\newcommand{\uExact}{\tilde{\uSDSys}}

\renewcommand{\v}{\mathbf{v}}
\newcommand{\vMacroOneD}{\bar{v}}
\newcommand{\vMacroMultiD}{\mathbf{\bar{v}}}
\newcommand{\vGrid}{\boldsymbol{v}}

\newcommand{\vmax}{v_{\max}}
\newcommand{\vstar}{v_{*}}
\newcommand{\x}{\mathbf{x}}
\newcommand{\xGrid}{\boldsymbol{x}}

\newcommand{\Cnm}[2]{\C^{#1 \times #2}}
\newcommand{\Rnm}[2]{\R^{#1 \times #2}}

\newcommand{\added}[1]{{\color{black}{#1}}}   
\newcommand{\moved}[1]{{\color{black}{#1}}}   
\newcommand{\changed}[1]{{\color{black}{#1}}} 

\newcommand{\Acal}{\boldsymbol{\mathcal{A}}}		
\newcommand{\f}{\mathbf{f}} 					
\newcommand{\Mv}{\Ma_v}					
\newcommand{\Mj}{\Ma_j}		 				
\newcommand{\Max}{\bm{\Ma}}					
\newcommand{\Mav}{\Max_v}					
\newcommand{\Maxj}{\Max_j}					
\newcommand{\Maxv}{\Max_\v}					
\renewcommand{\u}{\mathbf{u}} 				

\begin{document}

\title{A high-order relaxation method with projective integration for solving nonlinear systems of hyperbolic conservation laws}

\author{
Pauline Lafitte 
	\thanks{Laboratoire de Math\'ematiques Appliqu\'ees aux Syst\`emes, Ecole Centrale Paris, Grande Voie des Vignes, 92290 Ch\^atenay-Malabry, France ({\tt pauline.lafitte@ecp.fr}).} \and 
Ward Melis
	\thanks{Department of Computer Science, K.U. Leuven, Celestijnenlaan 200A, 3001 Leuven, Belgium ({\tt ward.melis@cs.kuleuven.be}).} \and
Giovanni Samaey
	\thanks{Department of Computer Science, K.U. Leuven, Celestijnenlaan 200A, 3001 Leuven, Belgium ({\tt giovanni.samaey@cs.kuleuven.be}).}}

\maketitle

\begin{abstract}	
	We present a general, high-order, fully explicit relaxation scheme \added{which can be applied to any system} of nonlinear hyperbolic conservation laws in multiple dimensions. The scheme consists of two steps. \added{In a first (relaxation) step,} the nonlinear hyperbolic conservation law is approximated by a kinetic equation with stiff BGK source term. Then, this kinetic equation is integrated in time using a projective integration method. After taking a few small (inner) steps with a simple, explicit method (such as direct forward Euler) to damp out the stiff components of the solution, the time derivative is estimated and used in an (outer) Runge-Kutta method of arbitrary order.  We show that, with an appropriate choice of inner step size, the time step restriction on the outer time step is similar to the CFL condition for the hyperbolic conservation law. Moreover, the number of inner time steps is also independent of the stiffness of the BGK source term. We discuss stability and consistency, and illustrate with numerical results (linear advection, Burgers' equation and the shallow water and Euler equations) in one and two spatial dimensions.
\end{abstract}

\section{Introduction} \label{sec:introduction}
Hyperbolic conservation laws arise in numerous physical applications, such as fluid dynamics, plasma physics, traffic modeling and electromagnetism (see, e.g., \cite{whitham2011linear,LeVeque2002}). They express the conservation of physical quantities (such as mass, momentum, or energy) and may be supplemented with boundary conditions that control influx or outflux at the boundaries of the physical domain \cite{LeVeque2002}. In this paper, we consider a system of hyperbolic conservation laws in multiple spatial dimensions:
\begin{equation} \label{eq:cons_law_system} 
	\partial_t \uSys + \nabla_{\x} \cdot \F(\uSys) = 0, 
\end{equation}
or, equivalently,
\begin{equation} \label{eq:cons_law_system_long} 
	\partial_t \uSys + \sum_{d=1}^{D} \partial_{x^d}\F^d(\uSys) = 0, 
\end{equation}
in which $\x=\set{x^d}{d} \in \R^D$ represents the space variables ($D$ being the number of spatial dimensions), $\uSys(\x,t) :=  \set{u_m(\x,t)}{m} \in \R^{M}$ denotes the conserved quantities, and $\F(\uSys) \in \Rnm{M}{D}$ corresponds to the flux functions.
 
Hyperbolic conservation laws are often solved using a finite volume method \cite{LeVeque2002,morton2005numerical}, which is derived from the integral expression of the conservation law. To this end, in a scalar one-dimensional setting and with a spatially uniform grid, the domain is divided in $I$ cells $\cell_i = [x_{i-1/2},x_{i+1/2}]$ with constant cell width $\dx$ over which the cell average of the solution $u(x,t)$ to the conservation law 
\begin{equation} \label{eq:cons_law_1d_scalar} 
	\partial_t u + \partial_x F(u)= 0,
\end{equation}	
is approximated at time $t=t^n$ by
\begin{equation} \label{eq:FV_average_solution}
	U_i^n \approx \frac{1}{\dx}\int_{\cell_i} u(x,t^n)dx. 
\end{equation}
Note that boldface is removed whenever the quantities are scalar. A numerical scheme is then constructed by integrating the conservation law \eqref{eq:cons_law_1d_scalar} in space over the cell $\cell_i$ and in time from $t^n$ to $t^{n+1}$ to obtain
\begin{equation} \label{FV_scheme} 
	U_i^{n+1} = U_i^n - \frac{\Dt}{\dx}\left(F_{i+1/2}^n - F_{i-1/2}^n\right), 
\end{equation}
in which $\Dt = t^{n+1} - t^n$ and the \emph{numerical flux} satisfies
\begin{equation} \label{eq:FV_numerical_flux}
	F_{i \pm 1/2}^n \approx \frac{1}{\Dt}\int_{t^n}^{t^{n+1}} F\left(u(x_{i \pm 1/2},t)\right)dt. 
\end{equation}
Clearly, equation \eqref{FV_scheme} is conservative by construction. The numerical fluxes $F_{i \pm 1/2}^n$ can be obtained by constructing an (approximate) Riemann solver, based on a (possibly high-order) reconstruction of the solution in each of the cells using interpolation over the neighboring cells \cite{LeVeque2002,eno_weno}. \moved{However, in the general nonlinear case, these spatial discretizations require the (possibly tedious) computation of the solutions of local Riemann problems.} 

\added{Relaxation methods offer an interesting alternative in which the nonlinear hyperbolic conservation law is replaced by a linear transport equation with a stiff nonlinear (but local) source term,} see, e.g., discrete kinetic schemes in \cite{Liu1987,Jin1998,Jin1995} and, in particular, \cite{Aregba-Driollet2000} which also contains a brief historical overview. In a relaxation method, the conservation law \eqref{eq:cons_law_system} is approximated by a problem of higher dimension containing a small relaxation parameter $\epsi$ such that, when $\epsi$ tends to zero, the original problem is recovered. In this paper, we will consider the relaxation problem to be a kinetic BGK equation. In a scalar one-dimensional setting, this equation describes the evolution of a distribution function $f^\epsi(x,v,t)$ of particles at position $x$ with velocity $v$ at time $t$ and takes the following form: 
\begin{equation} \label{eq:kin_eq_1d_scalar} 
	\partial_t f^\epsi + v\partial_x f^\epsi = \frac{1}{\epsi}\left(\Mv(u^\epsi) - f^\epsi\right).
\end{equation} 
The left hand side of equation \eqref{eq:kin_eq_1d_scalar} describes the transport of the particles whereas the right hand side represents the collisions between particles, which is modeled as a linear relaxation to the Maxwellian $\Mv(u^\epsi)$ with a relaxation time $\epsi$. \moved{The idea is that some of the difficulties associated with the original problem are avoided, while, for sufficiently small $\epsi$, the relaxation problem is a good approximation of the problem of interest.} \added{In particular,} the advantage of the kinetic equation \eqref{eq:kin_eq_1d_scalar} over the conservation law \eqref{eq:cons_law_1d_scalar} is the fact that the advection term in \eqref{eq:kin_eq_1d_scalar} is now linear, \added{removing the difficulties associated with the high-order discretization of a nonlinear flux term}. The disadvantage is the appearance of a stiff source term, which requires special care during time integration. The first methods, proposed in \cite{Jin1995,Aregba-Driollet2000} are based on splitting techniques. \added{As a consequence, the order in time is restricted to 2 and can only be improved by nontrivial manipulations, see \cite{Csomos2008}}. More recently, several \emph{asymptotic-preserving} methods based on IMEX techniques (in the sense of Jin \cite{Jin1999}) have been proposed that integrate the Boltzmann equation in the hyperbolic and diffusive regimes with a computational cost that is independent of $\epsi$ (see \cite{filbet2010class} and references within).  An appealing idea along this line of thought, based on IMEX Runge-Kutta methods, is presented in \added{\cite{Boscarino2009,Boscarino2013}}. Unfortunately, the proposed method is not very robust since it breaks down in intermediate \added{or hydrodynamic regimes}. An improvement was proposed in \cite{Dimarco2013asymptotic}.

\moved{In this paper, we propose to use a projective integration method to solve the stiff relaxation systems with an arbitrary order of accuracy in time. We will show that the resulting scheme constitutes \added{a flexible,} robust and fully explicit alternative to splitting and IMEX methods, while avoiding the construction of complicated and problem-specific (approximate) Riemann solvers. Projective integration methods were proposed in \cite{Gear2003projective} for stiff systems of ordinary differential equations and analyzed in \cite{Lafitte2012} for kinetic equations with a diffusive scaling. An arbitrary order version, based on Runge-Kutta methods, has been proposed recently in \cite{Lejon2016}, where it was also analyzed for kinetic equations with an advection-diffusion limit.} \added{Projective integration is particularly suited for stiff problems with a clear spectral gap.} In such stiff problems, the fast modes, corresponding to the Jacobian eigenvalues with large negative real parts, decay quickly, whereas the slow modes correspond to eigenvalues of smaller magnitude and are the solution components of practical interest. Projective integration allows a stable yet explicit integration of such problems by first taking a few small (inner) steps with a simple, explicit method, until the transients corresponding to the fast modes have died out, and subsequently projecting (extrapolating) the solution forward in time over a large (outer) time step. \added{Besides being robust and fully explicit, the resulting projective integration relaxation method is very appealing for nonlinear hyperbolic conservation laws because of its flexibility: once a solver is available, applying it to a different nonlinear hyperbolic conservation law merely amounts to changing the definition of the Maxwellian function $\Mv(u)$ in equation~\eqref{eq:kin_eq_1d_scalar}, leaving both the space and time discretizations untouched.} 

Projective integration fits within recent research efforts on numerical methods for multiscale simulation \cite{E2003a,Kevrekidis2003,Kevrekidis2009}.
\added{In this context, projective integration is a useful technique to effectively deal with problems in which there is a macroscopic (slow) dynamics whose mathematical formulation is not known and that can be captured ``on-the-fly'' by a short (appropriately initialized) microscopic simulation. Then, a few small steps of the full microscopic dynamics are combined with an extrapolation of the macroscopic, slow degrees of freedom only, and the resulting method is called \emph{coarse projective integration}. Examples are, amongst others, bacterial chemotaxis \cite{Setayeshgar2005}, chemical reactions \cite{Rico-Martinez} and disease modeling \cite{Cisternas2004}. For more examples, we refer to \cite{Kevrekidis2009}. To conclude, we also mention alternative approaches to obtain a higher-order projective integration scheme which have been proposed in \cite{Lee2007,Rico-Martinez}; see also \cite{Eriksson2004,Sommeijer1990,Vandekerckhove2007} for related work.}

The remainder of this paper is structured as follows. In section \ref{sec:kinetic}, we introduce the kinetic equations that form the basis of the relaxation method, and discuss their asymptotic equivalence with the original hyperbolic problem. In section \ref{sec:proj_int}, we describe the projective integration method that will be used to integrate these kinetic equations. We then \added{analyze} convergence of the resulting projective integration relaxation method for hyperbolic conservation laws in section \ref{sec:stability}, including the choice of appropriate method parameters. \added{This analysis is based on the results in \cite{Lejon2016}, for which we provide a number of alternative, simplified proofs that are specific for the relaxation systems of section~\ref{sec:kinetic}.} Section \ref{sec:results} reports the results of extensive numerical tests for a set of benchmark problems in both one and two space dimensions: linear advection, nonlinear conservation, the dam-break problem and Sod's shock test. We conclude in section \ref{sec:conclusions} with a brief discussion and some ideas for future work.
 
\section{Relaxation systems} \label{sec:kinetic}

\subsection{Kinetic equation and hydrodynamic limit}
To solve equation \eqref{eq:cons_law_system}, we introduce, as in \cite{Aregba-Driollet2000}, the (hyperbolically scaled) kinetic equation 
\begin{equation} \label{eq:kin_eq_system} 
	\partial_t \f^\epsi + \v\cdot\nabla_{\x}\cdot \f^\epsi = \frac{1}{\epsi}(\Maxv(\u^\epsi) - \f^\epsi),
\end{equation}
or, equivalently, 
\begin{equation} \label{eq:kin_eq_system_long} 
	\partial_t\f^\epsi + \sum_{d=1}^{D} v^d\partial_{x^d} \f^\epsi = \frac{1}{\epsi}(\Maxv(\u^\epsi) - \f^\epsi), 
\end{equation}
modeling the evolution of a vector of particle distribution functions $\f^\epsi(\x,\v,t) = (f^\epsi_m(\x,\v,t))_{m=1}^{M} \in \R^{M}$. The particle positions and velocities are represented as $\x = \set{x^d}{d} \in \R^D$ and $\v = \set{v^d}{d} \in V \subset \R^D$, respectively, and the right hand side of \eqref{eq:kin_eq_system_long} represents a BGK collision operator \cite{Bhatnagar1954}, modeling linear relaxation of $\f^\epsi$ to a Maxwellian distribution $\Maxv(\uSys^\epsi) \in \R^{M}$, in which the argument ${\uSys^\epsi(\x,t) = \averageV{\f^\epsi(\x,\v,t)}}$ is the density, obtained via averaging over the measured velocity space $(V,\mu)$, 
\begin{equation} \label{eq:V_average}
	\uSys := \averageV{\f} = \int_V \f d\mu(\v).
\end{equation} 

The advantage of this kinetic formulation is that the advection term is now linear, and therefore easier to discretize. The disadvantage is the increased dimension, as well as the introduction of the stiff source term of size $O(1/\epsi)$. The projective integration scheme that we will propose in section \ref{sec:proj_int} allows to integrate this stiff source term using an explicit method of arbitrary order.

To ensure that the kinetic equation \eqref{eq:kin_eq_system_long} converges to the conservation law \eqref{eq:cons_law_system} in the hydrodynamic limit $\hydroLimit$, one requires 
\begin{equation} \label{eq:maxwellian_conditions}
	\begin{dcases}
      \averageV{\Maxv(\uSys)} = \uSys, \\
      \averageV{v^d \Maxv(\uSys)} = \F^d(\uSys), \qquad 1 \le d \le D. 
	\end{dcases}
\end{equation}
Then, one can show \cite{Aregba-Driollet2000} that, in the limit of $\hydroLimit$, the kinetic model \eqref{eq:kin_eq_system_long} is approximated by the following equation: 
\begin{equation} \label{eq:kin_eq_limit}
	\partial_t \uSys^\epsi + \nabla_{\x}\cdot \F(\uSys^\epsi) = \epsi \nabla_{\x}\cdot (\diffMatrix\nabla_{\x}\uSys^\epsi), 
\end{equation}
or, equivalently,
\begin{equation} \label{eq:kin_eq_limit_long}
	\partial_t \uSys^\epsi + \sum_{d=1}^{D}\partial_{x^d} \F^d(\uSys^\epsi)= \epsi \sum_{d=1}^D\partial_{x^d} \left(\sum_{d'=1}^D\diffMatrix_{dd'}\partial_{x^{d'}}\uSys^\epsi\right), 
\end{equation}
with the diffusion matrix $\diffMatrix$ given as
\begin{equation} \label{eq:diffusion_matrix}
	\diffMatrix_{dd'}(\uSys) := \averageV{v^dv^{d'}\partial_{\uSys} \Maxv(\uSys)} - \partial_{\uSys} \F^{d}\partial_\uSys \F^{d'},
\end{equation}
in which the $M \times M$ matrices $\partial_{\uSys} \Maxv(\uSys)$ and $\partial_{\uSys} \F^{d}$ represent the Jacobian matrices of $\Maxv(\uSys)$ and $\F(\uSys)$, respectively.

Clearly, equation \eqref{eq:kin_eq_system_long} is consistent with equation \eqref{eq:cons_law_system} to order $1$ in $\epsi$. 
Moreover, the analysis reveals an additional condition on $\Ma_\v$ and $V$. Indeed, to ensure the parabolicity of \eqref{eq:kin_eq_limit_long}, the diffusion matrix $\diffMatrix$ should be positive definite. This leads to the so-called \emph{subcharacteristic condition} \cite{Bouchut1999,Aregba-Driollet2000},
\begin{equation} \label{eq:subcharacteristic_condition} 
	\sum_{d,d'=1}^D \left(\diffMatrix_{dd'}(\uSys) \xi^{d'}\cdot\xi^d\right) \ge 0, 
\end{equation}
for all $\xi^d$, $1\le d \le D$ in $\R^{M}$.

In what follows, we will always assume that the velocity space is discrete and of the form
\begin{equation} \label{eq:V_definition}
	V := \{\v_{j}\}_{j=1}^{J}, \qquad d\mu(\v)=\sum_{j=1}^J w_j \delta(\v-\v_j)d\v,
\end{equation}
with $\v_j$ denoting the chosen velocities and $w_j$ the corresponding weights. Due to this choice of $V$ the kinetic equation \eqref{eq:kin_eq_system} breaks up into a system of $J$ coupled partial differential equations,
\begin{equation} \label{eq:system_part_diff_eq} 
	\partial_t \f^\epsi_j + \v_j\cdot\nabla_{\x}\cdot \f^\epsi_j = \frac{1}{\epsi}(\Maxj(\uSys^\epsi) - \f^\epsi_j), \qquad  1 \le j \le J, 
\end{equation}
in which $\f^\epsi_j(\x,t) \equiv \f^\epsi(\x,\v_j,t)$, and the only coupling between different velocities is through the computation of $\uSys^\epsi$. As $\hydroLimit$, a Chapman-Enskog expansion allows to write
\begin{equation} \label{eq:chapman_enskog_first_order}
	\f^\epsi_j = \Maxj(\uSys^\epsi) +O(\epsi),
\end{equation}
so that, injecting it in \eqref{eq:system_part_diff_eq} and taking the mean value over $V$, we get
\begin{equation*} 
	\partial_t \averageV{\Maxj(\uSys^\epsi)}+\nabla_{\x}\cdot \averageV{\v_j\cdot\Maxj(\uSys^\epsi)} = O(\epsi).
\end{equation*}
Finally, the compatibility conditions \eqref{eq:maxwellian_conditions} imply
\begin{equation} \label{eq:system_part_diff_eq_limit_long}
	\partial_t \uSys^\epsi + \sum_{d=1}^D \partial_{x^d} \F^d(\uSys^\epsi) = O(\epsi).
\end{equation}

\begin{rem}[Minimal number of velocities] \label{rem:J2}
	In the relaxation problem \eqref{eq:system_part_diff_eq}, the minimal number $J$ of discrete velocities depends on the spatial dimension of the problem. \added{In particular, one needs to ensure -- at least -- that there is a velocity associated to each possible direction of motion. There are $J=2$ possible velocity directions in 1D (left, right) and $J=4$ possible directions in 2D (left, right, up, down)}. Since the computational complexity of the relaxation system is proportional to $J$, we will not consider higher values of $J$ here. Hence, the one-dimensional stability analysis in section~\ref{sec:stability} will be performed specifically for $J=2$, and in the experiments in section \ref{sec:results}, we will only use $J=2$ in 1D and $J=4$ in 2D. Finally, we stress that when choosing $J=2$ in 1D, the form of the discrete kinetic system given in equation \eqref{eq:system_part_diff_eq} coincides precisely with the relaxation system introduced by Jin and Xin in \cite{Jin1995}.
\end{rem}

\added{
\begin{rem}[Choice of $\epsi$] 
	From equation~\eqref{eq:system_part_diff_eq_limit_long}, it is clear that the relaxation system~\eqref{eq:system_part_diff_eq} contains a modeling error that is proportional to $\epsi$. In our projective integration schemes, however, a finite value of $\epsi$ will need to be chosen.  This choice will also be dictated by the numerical schemes; in particular, roundoff errors will become important as $\hydroLimit$, see section~\ref{subsec:consist}.
\end{rem}
}

\subsection{One-dimensional examples} \label{subsec:1d}
In one space dimension, we write equation \eqref{eq:cons_law_system} as 
\begin{equation} \label{eq:cons_law_1d_system}
	\partial_t \uSys + \partial_x \F(\uSys)= 0,
\end{equation}
in which $t \ge 0$ (resp. $x \in \R$) represents the time (resp. space) variable, ${\uSys(x,t) = \set{u_m(x,t)}{m} \in \R^{M}}$ embodies the conserved quantities, and $\F(\uSys) = \set{F_m(\uSys)}{m} \in \R^{M}$ denotes the flux functions. Correspondingly, the kinetic equation \eqref{eq:kin_eq_system_long} becomes
\begin{equation} \label{eq:kin_eq_1d} 
	\partial_t \f^\epsi + v \partial_x \f^\epsi = \frac{1}{\epsi}(\Mav(\uSys^\epsi) - \f^\epsi),
\end{equation}
with the particle distribution function $\f^\epsi(x,v,t) =\set{f^\epsi_m(x,v,t)}{m} \in \R^{M}$, and the particle velocities represented as $v \in V \subset \R$.

In equation \eqref{eq:kin_eq_1d}, the Maxwellian $\Mav$ is as yet not completely defined, as we only require the conditions \eqref{eq:maxwellian_conditions} to be satisfied. 
A physically relevant Maxwellian that corresponds to these conditions, is
\begin{equation} \label{eq:maxwellian_realistic_1d} 
	\Mav(\uSys^\epsi) = \uSys^\epsi + \frac{v\F(\uSys^\epsi)}{\averageV{v^2}},
\end{equation}
in which $V=\R$ with measure
\begin{equation} \label{eq:V_measure_gaussian} 
	d\mu(v) = \frac{1}{\sqrt{2\pi\sigma^2}}\exp\left(-\frac{v^2}{2\sigma^2}\right)dv.
\end{equation}
To satisfy the subcharacteristic condition given in \eqref{eq:subcharacteristic_condition}, the variance $\sigma^2 \in \R^+$ of the measure $d\mu(v)$ needs to be chosen appropriately. \added{For instance, in the scalar case, using equations \eqref{eq:diffusion_matrix} and \eqref{eq:maxwellian_realistic_1d} we obtain the following constraint:
\[ \averageV{v^2\left(1 + \frac{vF'(u)}{\averageV{v^2}}\right)} \ge (F'(u))^2. \]
Since this velocity space is odd symmetric, meaning that $\int_V h(v)d\mu(v) = 0$ for every odd function $h(v)$, this condition further reduces to:
\[ \averageV{v^2} \ge (F'(u))^2. \]
For the Gaussian measure in equation \eqref{eq:V_measure_gaussian} we have $\averageV{v^2} = \sigma^2$. Consequently, we require ${\sigma \ge \max_u\abs{F'(u)}}$ to ensure parabolicity of equation \eqref{eq:kin_eq_limit_long}.} Based on \eqref{eq:V_measure_gaussian}, we choose a discrete measured \added{symmetric} velocity space with an even number $J$ of velocities that satisfy $v_{J-j+1}\equiv -v_j$.
From the measure given in \eqref{eq:V_measure_gaussian}, these discrete velocities $\set{v_j}{j}$ and weights $\set{w_j}{j}$ are derived as the nodes and weights of the corresponding Gauss-Hermite quadrature. In particular, for $J=2$, this results in $v_{j}=\pm\sigma$, with corresponding weights $w_j=1/2$.

An alternative suggestion, that is equivalent to the above choice for $J=2$ and $\sigma=1$, was proposed in \cite{Bouchut1999}. There, a Maxwellian of the form
\begin{equation} \label{eq:maxwellian_artificial_1d} 
	\Mav(\uSys^\epsi) = \uSys^\epsi + \frac{\F(\uSys^\epsi)}{v} 
\end{equation}
is proposed. For both choices \eqref{eq:maxwellian_realistic_1d} and \eqref{eq:maxwellian_artificial_1d}, the conditions \eqref{eq:maxwellian_conditions} can be seen to be satisfied. In case of the Maxwellian given in equation \eqref{eq:maxwellian_artificial_1d}, the specific values of the velocities $v_j$ need to be chosen such that the subcharacteristic condition \eqref{eq:subcharacteristic_condition} is satisfied. When we further restrict to a scalar case, i.e., $M=1$,  
\begin{equation} \label{eq:cons_law_1d_scalar_again} 
	\partial_t u + \partial_x F(u)= 0, 
\end{equation}
\added{and choosing an odd symmetric velocity space, the subcharacteristic condition again gives rise to the condition $\averageV{v^2} \ge (F'(u))^2$, which is always satisfied when choosing the discrete velocities as:}
\begin{equation} \label{eq:subcharacteristic_condition_artificial} 
	\added{\abs{v_j} \ge \max_u\abs{F'(u)}, \quad 1 \le j \le J.}
\end{equation}
\added{The corresponding weights are chosen as $w_j = 1/J$.} Note again that all boldfaced typesetting is removed for a scalar case.

For the numerical illustrations, we choose concretely the following examples:
\begin{ex} The scalar linear advection equation,
	\begin{equation} \label{eq:flux_linear_advection_1d_scalar} 
		F(u) = a\cdot u, \qquad a \in \R. 
	\end{equation}
\end{ex}
\begin{ex} The scalar Burgers' equation,
	\begin{equation} \label{eq:flux_burgers_1d_scalar}
		F(u)=u^2/2.
	\end{equation}
\end{ex}
\begin{ex} The one-dimensional Euler equations,
	\begin{align} 
		\uSys &= (\rho, \rho\vMacroOneD, E), \label{eq:euler_conserved_vars} \\
	  	\F(\uSys) &= (\rho\vMacroOneD, \rho\vMacroOneD^2+P, \added{(E+P)\vMacroOneD}), \label{eq:euler_1d}
	\end{align}
	with the equation of state	
	\begin{equation} \label{eq:euler_eq_state_1d}
		P =  (\gamma-1)\left(E-\frac{1}{2}\rho\vMacroOneD^2\right).
	\end{equation}
\end{ex}

\subsection{Two-dimensional examples} \label{sec:2D}
In two space dimensions, we write equation \eqref{eq:cons_law_system} as 
\begin{equation} \label{eq:cons_law_2d_system}
	\partial_t \uSys + \partial_{x} \F^x(\uSys) + \partial_{y} \F^y(\uSys)= 0, 
\end{equation}
where $t \ge 0$ (resp. $x,y \in \R$) represent the time (resp. space) variables, ${\uSys(x,y,t) = (u_m(x,y,t))_{m=1}^{M} \in \R^{M}}$ correspond to the conserved quantities, and $\F^{x,y}(\uSys)=(F^{x,y}_m(\uSys))_{m=1}^{M}\in\R^{M}$ denote the fluxes in the $x$ and $y$ direction, respectively. Correspondingly, the kinetic equation \eqref{eq:kin_eq_system_long} becomes:
\begin{equation} \label{eq:kin_eq_2d} 
	\partial_t \f^\epsi + v^x \partial_x \f^\epsi + v^y \partial_y \f^\epsi = \frac{1}{\epsi}(\Maxv(\uSys^\epsi) - \f^\epsi),
\end{equation}
in which the particle distribution function $\f^\epsi(x,y,v^x,v^y,t)= (f^\epsi_m(x,y,v^x,v^y,t))_{m=1}^{M} \in \R^{M}$ and the particle velocities $\v=(v^x,v^y) \in V \subset \R^2$, with $v^{x,y}$ the velocity of the particles in the $x$ and $y$ direction, respectively. 

Compared to the one-dimensional setting, the choice of the Maxwellian and the description of the discrete velocity space are considerably more elaborate, and many options have been documented, see, e.g., \cite{Aregba-Driollet2000,Mieussens2000,Bobylev2003}. In the numerical examples in this paper, we choose the orthogonal velocities method, see, e.g., \cite{Aregba-Driollet2000}, which we now detail for the scalar case ($M=1$). In this method, we choose a set of velocities with varying length and direction. Specifically, we first fix a maximal velocity length $\vmax$. We then consider $R$ different velocity lengths:
\[ \rho_r = \dfrac{r}{R}\vmax, \qquad 1\le r \le R, \]
and $4S$ different velocity directions 
\[ \theta_s = \frac{s}{S}\frac{\pi}{2}, \qquad 1\le s \le 4 S, \]
with $R,S \ge 1$. We then obtain $J=4RS$ velocities $\v_j=(v^x_j,v^y_j)$, $1\le j \le J$, by assigning an index $j=\left(r-1\right)4S+s$ to every length-direction pair $(r,s)$, and writing
\begin{equation} \label{eq:ovm_vx_vy}
	v^x_j = \rho_r\cos(\theta_s), \qquad v^y_j = \rho_r\sin(\theta_s), \qquad 1 \le r \le R, \quad 1 \le s \le 4S.
\end{equation}

The Maxwellian function $\Ma_j$ for the $j^{th}$ equation of system \eqref{eq:system_part_diff_eq} is then chosen as:
\begin{equation} \label{eq:maxwellian_2d} 
	\Mj(u) = u +  v^x_j\frac{F^x(u)}{\averageV{(v^x)^2}} + v^y_j\frac{F^y(u)}{\averageV{(v^y)^2}}, 
\end{equation}
It can be shown that for the orthogonal velocities method we have $\averageV{(v^x)^2} = \averageV{(v^y)^2}$.

The generalization to $M > 1$ is straightforward. In \cite{Aregba-Driollet2000} it is proven that, for stability reasons, one should choose $\vmax$ as follows:
\begin{equation} \label{eq:vmax_2d} 
	\vmax^2 \ge \frac{12R^2\left(\norm{\partial_\uSys \F^x}^2 + \norm{\partial_\uSys \F^y}^2\right)}{(R+1)(2R+1)}, 
\end{equation} 
where $\norm{\cdot}$ is the matrix norm associated with the classical 2-norm when $M > 1$.

\begin{rem} \label{rem:RS1} 
	As pointed out in remark \ref{rem:J2}, in the experiments in section \ref{sec:results} we will always choose $J=4$ velocities in 2D. This is accomplished by setting $R=S=1$ in the orthogonal velocities method.
\end{rem}


\section{Projective integration} \label{sec:proj_int}
In this section, we construct a fully explicit, asymptotic-preserving, arbitrary order time integration method for the stiff system \eqref{eq:system_part_diff_eq}. The asymptotic-preserving property \cite{Jin1999} implies that, in the limit when $\epsi$ tends to zero, an $\epsi$-independent time step constraint, of the form $\Dt = O(\dx)$, can be used, \changed{in agreement with the classical hyperbolic CFL constraint for the limiting equation \eqref{eq:kin_eq_limit_long}}. To achieve this, we will use a projective integration method \cite{Gear2003projective,Lafitte2012}, which combines a few small time steps with a naive (\emph{inner}) timestepping method, such as a direct forward Euler discretization, with a much larger (\emph{projective, outer}) time step. The idea is sketched in figure \ref{fig:proj_int}.

\begin{figure}[t]
	\begin{center}
		\figname{sketch_PI}
		\begin{tikzpicture}[node distance = 2cm, auto]
	\tikzstyle{red_dot} = [red,fill=red]
	\draw [axis] (0,0) -- (9,0);
	\draw [axis] (0,0) -- (0,3);
	\node[align=center, right] at (9,0) (time) {time};
	
	\draw [thick] (0.3,0.3) to [out=45,in=145] (8,0.5);
	
	\draw[red_dot] (0.5,2) circle (.5ex);
	\draw [dotted] (0.5,0) -- (0.5,2); \node[below] at (0.5,0) (time) {$t^{n-1}$};
	\draw[red_dot] (0.8,1.25) circle (.5ex); \draw[red_dot] (1.1,0.98) circle (.5ex); \draw[red_dot] (1.4,1.15) circle (.5ex);
	\draw [dashed] (1.1,0.98) -- (4,2.8);
	
	\draw[red_dot] (4,2.8) circle (.5ex);
	\draw [dotted] (4,0) -- (4,2.8); \node[below] at (4,0) (time) {$t^{n}$};
	\draw[red_dot] (4.3,2.1) circle (.5ex); \draw[red_dot] (4.6,1.83) circle (.5ex); \draw[red_dot] (4.9,1.77) circle (.5ex);
	\draw [dashed] (4.6,1.83) -- (7.5,1.4);
	
	\draw[red_dot] (7.5,1.4) circle (.5ex);
	\draw [dotted] (7.5,0) -- (7.5,1.4); \node[below] at (7.5,0) (time) {$t^{n+1}$};

	\path
	([shift={(-5\pgflinewidth,-5\pgflinewidth)}]current bounding box.south west)
	([shift={( 2\pgflinewidth, 2\pgflinewidth)}]current bounding box.north east);
\end{tikzpicture}
	\end{center}
  	\vspace{-0.4cm}\caption{\label{fig:proj_int} Sketch of projective integration. At each time instance ($t^n$), an explicit method is applied over a number of small time steps (red dots) so as to stably integrate the fast modes. As soon as these modes are sufficiently damped the solution is extrapolated using a much larger time step (dashed lines). }
\end{figure}
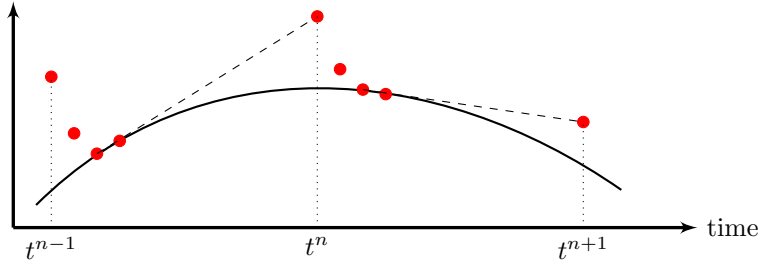

The inner and outer integrator can be selected independently. In section \ref{subsec:inner}, we discuss the inner integrator. Afterwards, in section \ref{subsec:outer}, we consider the outer integrator, before studying their numerical properties in section \ref{sec:stability}.  

\subsection{Inner integrators} \label{subsec:inner}
We intend to integrate \eqref{eq:system_part_diff_eq} on a uniform, constant in time, periodic spatial mesh with spacing $\dx$, consisting of $I$ mesh points $x_i=i\dx$, $1 \le i \le I$, with $I\dx=1$, and a uniform time mesh with time step $\dt$, i.e., $t^k=k\dt$. (Generalization to space-time adaptive grids is of course straightforward.) The numerical solution on this mesh is denoted as $\fSDSys_{i,j}^k$, where we have dropped the dependence on $\epsi$ for conciseness. After discretizing in space, we obtain a semi-discrete system of ordinary differential equations
\begin{equation}\label{eq:semidiscrete} 
	\dot{\fSDSys} = \Dtime(\fSDSys),  \qquad \Dtime(\fSDSys):=- \Dxv(\fSDSys) + \frac{1}{\epsi}\left(\Maxv(\uSys) - \fSDSys\right),
\end{equation}
where $\Dxv(\cdot)$ represents a suitable discretization of the first order spatial derivative $v\partial_x$ (e.g., upwind differences).

As inner integrator, we choose an explicit scheme, for which we will, later on, use the shorthand notation
\begin{equation} \label{eq:timestepper}
	\fSDSys^{k+1} = \Sdt(\fSDSys^{k}),\qquad k = 0, 1, \ldots 
\end{equation}
The forward Euler (FE) method and Runge-Kutta methods immediately come to mind.

\paragraph{Forward Euler (FE).} The simplest time discretization routine is the forward Euler method, 
\begin{equation} \label{eq:FE_scheme} 
	\fSDSys^{k+1} = \fSDSys^k + \dt\Dtime(\fSDSys^k). 
\end{equation}

\paragraph{Higher-order Runge-Kutta methods.} \label{sec:inner_RK2} 
To obtain higher-order accuracy in time in the inner integrator, one could also employ any Runge-Kutta method \cite{Wanner1991,Hairer1993}, such as the second order method,
\begin{align} \label{eq:RK2_scheme} 
	\mathbf{k}_1 &= \Dtime(\fSDSys^k), \\
  	\mathbf{k}_2 &= \Dtime\left(\fSDSys^k + \frac{\dt}{2}\mathbf{k}_1\right), \\
	\fSDSys^{k+1}& = \fSDSys^k + \dt\mathbf{k}_2. 
\end{align}

\added{In section \ref{sec:stability}, we will show that higher-order inner Runge-Kutta methods have spectral properties that make them unsuitable for use in conjunction with projective integration. Hence, in the following sections, we will always use forward Euler as the inner integrator.}

\subsection{Outer integrators} \label{subsec:outer}
In equation \eqref{eq:system_part_diff_eq}, the small parameter $\epsi$ in the relaxation term leads to the classical time step restriction of the form $\dt = O(\epsi)$ for the inner integrator. However, as $\epsi$ goes to $0$, we obtain the limiting equation \eqref{eq:system_part_diff_eq_limit_long} for which a standard finite volume/forward Euler method only needs to satisfy a stability restriction of the form $\Dt \le C\dx$, with $C$ a constant that depends on the specific choice of the scheme and the parameters of the equation.

In \cite{Lafitte2012}, it was proposed to use a projective integration method to accelerate such a brute-force integration; the idea, originating from \cite{Gear2003projective}, is the following. Starting from a computed numerical solution $\fSDSys^n$ at time $t^n=n\Dt$, one first takes $K+1$ \emph{inner} steps of size $\dt$,
\begin{equation} \label{eq:timestepper_in_pi}
	\fSDSys^{n,k+1} = \Sdt(\fSDSys^{n,k}), \qquad 0 \le k \le K,
\end{equation}
in which the superscript pair $(n,k)$ represents the numerical solution at $t^{n,k}=n\Dt +k\dt$. The aim is to obtain a discrete derivative to be used in the \emph{outer} step to compute $\fSDSys^{n+1} = \fSDSys^{n+1,0}$ via extrapolation in time, e.g.,
\begin{equation} \label{eq:PFE_scheme} 
	\fSDSys^{n+1} = \fSDSys^{n,K+1} + (\Dt - (K + 1)\dt)\frac{\fSDSys^{n,K+1} - \fSDSys^{n,K}}{\dt}.
\end{equation}
This method is called projective forward Euler, and it is the simplest instantiation of this class of integration methods \cite{Gear2003projective}.  

\changed{Higher-order projective integration methods can be constructed by replacing each time derivative evaluation $\mathbf{k}_s$ in a classical Runge-Kutta method by $K+1$ steps of an inner integrator as follows (with $\fSDSys^{n,0} = \fSDSys^n$ for consistency) \cite{Lejon2016}}:
\begin{align}
	s = 1 : & 
	\begin{dcases} 
		\fSDSys^{n,k+1} &= \fSDSys^{n,k} + \dt\Dtime(\fSDSys^{n,k}), \qquad 0 \le k \le K \\ 
		\mathbf{k}_1 &= \dfrac{\fSDSys^{n,K+1} - \fSDSys^{n,K}}{\dt}
	\end{dcases} \label{eq:PRK_stage_1} \\
	2 \le s \le S :& 
	\begin{dcases} 
		\fSDSys^{n+c_s,0}_s &= \fSDSys^{n,K+1} + (c_s\Dt-(K+1)\dt) \sum_{l=1}^{s-1}\dfrac{a_{s,l}}{c_s} \mathbf{k}_l, \\
		\fSDSys^{n+c_s,k+1}_s &= \fSDSys^{n+c_s,k}_s + \dt\Dtime(\fSDSys^{n+c_s,k}_s), \qquad 0 \le k \le K \\
		\mathbf{k}_s &= \dfrac{\fSDSys^{n+c_s,K+1}_s - \fSDSys^{n+c_s,K}_s}{\dt}
	\end{dcases} \label{eq:PRK_stage_s} \\
	& \fSDSys^{n+1} = \fSDSys^{n,K+1} + (\Dt-(K+1)\dt)\sum_{s=1}^{S}b_s \mathbf{k}_s.
\end{align}
To ensure consistency, the RK matrix $\mathbf{a}=(a_{s,l})_{s,l=1}^S$, weights $\mathbf{b}=(b_s)_{s=1}^S$, and nodes $\mathbf{c}=(c_s)_{s=1}^S$ satisfy (see, e.g., \cite{Hairer1993}) the conditions $0\le b_s \le 1$ and $0 \le c_s \le 1,$ as well as
\begin{equation} \label{eq:RK_conditions}
	\sum_{s=1}^Sb_s=1, \qquad \sum_{l=1}^{S-1} a_{s,l} =c_s, \quad 1 \le s \le S. 
\end{equation}
(Note that these assumptions imply that $c_1=0$ using the convention that $\sum_{1}^0\cdot=0$.)

In the numerical experiments, we will specifically use projective Runge-Kutta methods of orders 2 and 4, represented by the Butcher tableaux in figure \ref{tab:butcher}.

\begin{figure}[t]
	\begin{center}
		\figname{butcher_tableaux}
		\newcommand{\butcherVerSpacing}{0.6}
\newcommand{\butcherHorSpacing}{0.8}
\newcommand{\butcherMidXstart}{4}
\newcommand{\butcherRightXstart}{9}
\begin{tikzpicture}
	\draw[](-\butcherHorSpacing/2,\butcherVerSpacing/2) -- (1.5*\butcherHorSpacing+0.15,\butcherVerSpacing/2);
	\draw[](\butcherHorSpacing/2,-\butcherVerSpacing/2) -- (\butcherHorSpacing/2,1.5*\butcherVerSpacing);
	
	\node(c) at (0,\butcherVerSpacing) {$\mathbf{c}$};
	\node(a) at (\butcherHorSpacing,\butcherVerSpacing) {$\mathbf{a}$};
	\node(bT) at (\butcherHorSpacing+0.15,0) {$\mathbf{b}^T$};
	
	\draw[] (\butcherMidXstart-\butcherHorSpacing/2,\butcherVerSpacing/2) -- (\butcherMidXstart+2.5*\butcherHorSpacing,\butcherVerSpacing/2);
	\draw[] (\butcherMidXstart+\butcherHorSpacing/2,-\butcherVerSpacing/2) --(\butcherMidXstart+\butcherHorSpacing/2,2.5*\butcherVerSpacing);
	
	\node(c1_sec) at (\butcherMidXstart,\butcherVerSpacing) {$1/2$};
	\node (c2_sec) at (\butcherMidXstart,2*\butcherVerSpacing){$0$};
	\node (a21_sec) at (\butcherMidXstart+\butcherHorSpacing,\butcherVerSpacing) {$1/2$};
	\node (b1_sec) at (\butcherMidXstart+\butcherHorSpacing,0) {$0$};
	\node (b2_sec) at (\butcherMidXstart+2*\butcherHorSpacing,0) {$1$};
	
	\draw[] (\butcherRightXstart-\butcherHorSpacing/2,\butcherVerSpacing/2) -- (\butcherRightXstart+4.5*\butcherHorSpacing,\butcherVerSpacing/2);
	\draw[] (\butcherRightXstart+\butcherHorSpacing/2,-\butcherVerSpacing/2) -- (\butcherRightXstart+\butcherHorSpacing/2,4.5*\butcherVerSpacing);

	\node (c1_fourth) at (\butcherRightXstart,\butcherVerSpacing) {$1$};	
	\node (c2_fourth) at (\butcherRightXstart,2*\butcherVerSpacing) {$1/2$};
	\node (c3_fourth) at (\butcherRightXstart,3*\butcherVerSpacing) {$1/2$};
	\node (c4_fourth) at (\butcherRightXstart,4*\butcherVerSpacing) {$0$};
	\node (a41_fourth) at (\butcherRightXstart+\butcherHorSpacing,\butcherVerSpacing) {$0$};
	\node (a42_fourth) at (\butcherRightXstart+2*\butcherHorSpacing,\butcherVerSpacing) {$0$};
	\node (a43_fourth) at (\butcherRightXstart+3*\butcherHorSpacing,\butcherVerSpacing) {$1$};
	\node (a31_fourth) at (\butcherRightXstart+\butcherHorSpacing,2*\butcherVerSpacing) {$0$};
	\node (a32_fourth) at (\butcherRightXstart+2*\butcherHorSpacing,2*\butcherVerSpacing) {$1/2$};
	\node (a21_fourth) at (\butcherRightXstart+\butcherHorSpacing,3*\butcherVerSpacing) {$1/2$};
	\node(b1_fourth) at (\butcherRightXstart+\butcherHorSpacing,0) {$1/6$};
	\node (b2_fourth) at (\butcherRightXstart+2*\butcherHorSpacing,0) {$1/3$};
	\node (b3_fourth) at (\butcherRightXstart+3*\butcherHorSpacing,0) {$1/3$};
	\node (b4_fourth) at (\butcherRightXstart+4*\butcherHorSpacing,0) {$1/6$};
\end{tikzpicture}
	\end{center}
	\vspace{-0.4cm}\caption {\label{tab:butcher} Butcher tableaux for Runge-Kutta methods. Left: general notation; middle: RK2 method (second order); right: RK4 method (fourth order).}
\end{figure}

\subsection{Stability of projective integration} \label{sec:stab_pi}
We now briefly discuss the main stability properties of projective Runge-Kutta methods as derived in \cite{Lejon2016}. To this end, we introduce the test equation and its corresponding inner integrator,
\begin{equation} \label{eq:test_equation}
	\dot{y} = \lambda y, \qquad y^{k+1} = \tau(\lambda\dt)y^k, \qquad \lambda \in \C.
\end{equation}
As in \cite{Gear2003projective}, we call $\tau(\lambda\dt)$ the \emph{amplification factor} of the inner integrator. (For instance, if the inner integrator is the forward Euler scheme, we have $\tau(\lambda\dt) = 1+\lambda\dt$.) The inner integrator is stable if $\abs{\tau}\le 1$. The question then is for which subset of these values the projective integration method is also stable.  

Considering projective forward Euler, it can easily be seen from \eqref{eq:PFE_scheme} that the projective forward Euler method is stable if
\begin{equation}\label{eq:pfe_stab_cond}
	\abs{\left[\left(\dfrac{\Dt-(K+1)\dt}{\dt} + 1\right)\tau - \dfrac{\Dt-(K+1)\dt}{\dt}\right]\tau^K} \le 1,
\end{equation}
for all eigenvalues $\tau$ of the inner integrator for the kinetic equation \eqref{eq:kin_eq_1d}. The goal is to take a projective time step $\Dt = O(\dx)$, whereas $\dt = O(\epsi)$ necessarily to ensure stability of the inner brute-force forward Euler integration. Since we are interested in the limit $\hydroLimit$ for fixed $\dx$, we look at the limiting stability regions as $\Dt/\dt \to \infty$. In this regime, it is shown in
\cite{Gear2003projective} that the values $\tau$ for which the condition \eqref{eq:pfe_stab_cond} is satisfied lie in the union of two separated disks $\diskpfe_1 \cup \diskpfe_2$ where
\begin{equation}\label{eq:stab_pfe}
	\diskpfe_1=\disk\left(1-\dfrac{\dt}{\Dt},\dfrac{\dt}{\Dt}\right)\text{ and }
	\diskpfe_2=\disk\left(0,\left(\dfrac{\dt}{\Dt}\right)^{1/K}\right),
\end{equation}
and $\disk(c,r)$ denotes the disk with center $(c,0)$ and radius $r$. One then aims at positioning the eigenvalues that correspond to modes that are quickly damped by the time-stepper in $\diskpfe_2$, whereas the eigenvalues in $\diskpfe_1$ should correspond to slowly decaying modes. The projective integration method then allows for accurate integration of the modes in $\diskpfe_1$ while maintaining stability for the modes in $\diskpfe_2$.

\changed{In \cite{Lejon2016}, this analysis is extended to the projective Runge-Kutta case, showing that, in the limit when $\dt/\Dt$ tends to $0$, the stability region of the projective Runge-Kutta method also breaks up into two regions $\diskprk{$q$}_1$ and $\diskprk{$q$}_2$, which moreover satisfy
\[ \diskprk{$q+1$}_1 \supseteq \diskprk{$q$}_1 \supseteq \diskpfe_1 \text{ and } \diskprk{$q+1$}_2 \supseteq \diskprk{$q$}_2 \supseteq \diskpfe_2, \qquad \forall q \in \{1,2,...\}, \]
in which the constant $q$ indicates the order of the specific Runge-Kutta method. This implies that the stability regions of higher-order projective integration methods are contained within those of the lower-order ones. 
The main conclusion is that, whereas the stability regions of higher-order projective Runge-Kutta methods differ from those of projective forward Euler in their precise shape, their qualitative dependence on the parameters of projective integration ($\dt$, $K$ and $\Dt$) is identical, and method parameters that are suitable for projective forward Euler will also be suitable for the higher-order projective Runge-Kutta method.}

\vspace{0.3cm}
\section{Numerical properties} \label{sec:stability}
\added{Now we are ready to use the projective integration method on the relaxation system \eqref{eq:semidiscrete}. The parameters to determine are then the time scale separation parameter $\epsi$ in the relaxation system \eqref{eq:semidiscrete}, as well as the projective integration parameters: the inner time step $\delta t$, the outer time step $\Delta t$ and the number of inner steps $K$. } The projective integration parameters $\dt$, $K$, and $\Dt$ can be determined by imposing that all the eigenvalues of the selected inner integrator scheme fall into the stability region of the \changed{projective integration method}. \moved{While the numerical experiments also deal with systems of nonlinear hyperbolic conservation laws in multiple space dimensions, the analysis is restricted to a one-dimensional, scalar, linear setting.} In section \ref{subsec:inner_integrator_spectrum}, we will calculate the spectrum of the inner integrators constructed for the relaxation system given in \eqref{eq:semidiscrete}, and this in the specific case $J=2$ (see Remark \ref{rem:J2}). \changed{This result is a special case \added{(with an adapted proof)} of the more general result for any (even) number of velocities $J$ that was obtained in \cite{Lejon2016}:} for $J=2$, we are able to derive explicit asymptotic expansions for both the fast and slow eigenvalues in the spectrum. Then, we derive suitable choices for the projective integration parameters \added{in the specific setting of this paper} (section \ref{sec:param}). 
\added{Finally, we elaborate on the consistency of the resulting method for the hyperbolic conservation law~\eqref{eq:cons_law_1d_scalar} in section~\ref{subsec:consist}. Here, we will also see how to properly choose the relaxation parameter $\epsi$.}

\subsection{Spectrum of inner integrators} \label{subsec:inner_integrator_spectrum}
To compute bounds on the spectrum of the inner integrator for the kinetic equation \eqref{eq:kin_eq_system_long} with a linear Maxwellian 
\[ \Mv(u) = u + \dfrac{u}{v}, \]
we first rewrite the semi-discretized kinetic equation \eqref{eq:semidiscrete} in the (spatial) Fourier domain, 
\begin{equation}\label{eq:semidiscrete_fourier}
	\partial_t \fF(\zeta) = \Bcal\;\fF(\zeta), \qquad
	\Bcal = \dfrac{1}{\epsi}(-\epsi\DF + \MF\PF - \eye),
\end{equation}
with $\fF \in \C^2$, $\Bcal$, $\DF \in \Cnm{2}{2}$, $\MF$, $\PF \in \Rnm{2}{2}$, and $\eye$ the identity matrix of dimension $J=2$. In \eqref{eq:semidiscrete_fourier}, the matrix $\DF$ represents the (diagonal) Fourier matrix of the spatial discretization chosen for the convection part, $\PF$ is the Fourier matrix of the averaging of $f$ over the positive and negative velocities,
\begin{equation} \label{eq:P_fourier} 
	\PF:= \frac{1}{2}\begin{pmatrix} 1 & 1 \\ 1 & 1 \end{pmatrix},
\end{equation}
and the matrix $\MF$ represents the Fourier transform of the Maxwellian,
\begin{equation} \label{eq:M_fourier} 
	\MF = \eye +\mathbf{V}^{-1},
\end{equation}
with $\mathbf{V} = \diag{[\vstar,-\vstar]}$ and $\vstar$ the chosen (discrete) velocity component.

Since we are using a symmetric velocity space, we have the following property on the diagonal elements of the matrix $\DF$: $D_1 = \bar{D_2}$. Hence, from now on, we write the diagonal elements of $\DF$ as
\[ D_{1,2} = \alpha \pm \imath\beta, \]
in which $\alpha$ and $\beta$ depend on the spatial discretization, the velocity value $\vstar$ and the Fourier mode $\zeta$\added{, see table \ref{tab:upwind}.}

\begin{table}[t]
	\begin{center}
		\figname{upwind_coef}
		\newcommand{\uWidthZero}{0.8}
\newcommand{\uWidthOne}{3}
\newcommand{\uWidthTwo}{4.5}
\newcommand{\uWidthThree}{4.5}
\newcommand{\uHeightOne}{0.7}
\newcommand{\uHeightTwo}{1.5}
\newcommand{\uHeightThree}{1}
\newcommand{\uTotalHeight}{\uHeightOne+\uHeightTwo+\uHeightThree}
\newcommand{\uTotalWidth}{\uWidthZero+\uWidthOne+\uWidthTwo+\uWidthThree}
\newcommand{\col}{black}
\begin{tikzpicture}
	\draw[color=\col] (0,\uHeightThree) -- (\uTotalWidth,\uHeightThree);
	\draw[color=\col,line width=0.4mm] (0,\uHeightThree+\uHeightTwo) -- (\uTotalWidth,\uHeightThree+\uHeightTwo);
	\draw[color=\col,line width=0.4mm] (\uWidthZero,0) -- (\uWidthZero,\uTotalHeight);
	\draw[color=\col] (\uWidthZero+\uWidthOne,0) -- (\uWidthZero+\uWidthOne,\uTotalHeight);
	\draw[color=\col] (\uWidthZero+\uWidthOne+\uWidthTwo,0) -- (\uWidthZero+\uWidthOne+\uWidthTwo,\uTotalHeight);
	\draw[color=\col] (\uTotalWidth,0) -- (\uTotalWidth,\uTotalHeight);
	
	\node[color=\col] (alfa) at (\uWidthZero/2,\uHeightThree+\uHeightTwo/2) {$\alpha$};
	\node[color=\col] (beta) at (\uWidthZero/2,\uHeightThree/2) {$\beta$};
	\node[color=\col] (u1) at (\uWidthZero+\uWidthOne/2,\uHeightThree+\uHeightTwo+\uHeightOne/2) {upwind 1};
	\node[color=\col] (u1) at (\uWidthZero+\uWidthOne+\uWidthTwo/2,\uHeightThree+\uHeightTwo+\uHeightOne/2) {upwind 2};
	\node[color=\col] (u1) at (\uWidthZero+\uWidthOne+\uWidthTwo+\uWidthThree/2,\uHeightThree+\uHeightTwo+\uHeightOne/2) {upwind 3};
	
	\node[color=\col] (alfa_u1) at (\uWidthZero+\uWidthOne/2,\uHeightThree+\uHeightTwo/2) {$-\vstar\dfrac{2\sin^2\left(\dfrac{\zeta}{2}\right)}{\dx}$};
	\node[color=\col] (alfa_u2) at (\uWidthZero+\uWidthOne+\uWidthTwo/2,\uHeightThree+\uHeightTwo/2) {$-\vstar\dfrac{3-4\cos(\zeta)+\cos(2\zeta)}{2\dx}$};
	\node[color=\col] (alfa_u3) at (\uWidthZero+\uWidthOne+\uWidthTwo+\uWidthThree/2,\uHeightThree+\uHeightTwo/2) {$-\vstar\dfrac{3-4\cos(\zeta)+\cos(2\zeta)}{6\dx}$};
	\node[color=\col] (beta_u1) at (\uWidthZero+\uWidthOne/2,\uHeightThree/2) {$-\vstar\dfrac{\sin(\zeta)}{\dx}$};
	\node[color=\col] (beta_u2) at (\uWidthZero+\uWidthOne+\uWidthTwo/2,\uHeightThree/2) {$-\vstar\dfrac{4\sin(\zeta)-\sin(2\zeta)}{2\dx}$};
	\node[color=\col] (beta_u3) at (\uWidthZero+\uWidthOne+\uWidthTwo+\uWidthThree/2,\uHeightThree/2) {$-\vstar\dfrac{8\sin(\zeta)-\sin(2\zeta)}{6\dx}$};
\end{tikzpicture}
	\end{center}
	\vspace{-0.4cm}\caption {\label{tab:upwind} Dependence of $\alpha$ and $\beta$ on $\vstar$, $\zeta$ and the chosen spatial discretization technique.}
\end{table}

Since in the Fourier domain we are calculating the eigenvalues of $(2 \times 2)$ matrices, it is easy to prove the following theorem.
\begin{thm} \label{thm:spectrum}
	Under the above assumptions, the spectrum of the matrix $\Bcal = \dfrac{1}{\epsi}(\MF\PF - \eye - \epsi\DF)$ contains one slow eigenvalue $\lambda_1$ and one fast eigenvalue $\lambda_2$ which can be expanded as,
	\begin{equation} \label{eq:spectrum_B}
		\begin{aligned}
		\lambda_1 &= -\alpha + \beta^2\left(\frac{1 - \vstar^2}{\vstar^2}\right)\epsi + O(\epsi^3) + \imath\left(-\frac{\beta}{\vstar} + 2\frac{\beta^3}{\vstar}\left(\frac{1-\vstar^2}{\vstar^2}\right)\epsi^2 + O(\epsi^4)\right)& \\
		\lambda_2 &= -\frac{1}{\epsi} - \alpha - \beta^2\left(\frac{1 - \vstar^2}{\vstar^2}\right)\epsi + O(\epsi^3) + \imath\left(\frac{\beta}{\vstar} - 2\frac{\beta^3}{\vstar}\left(\frac{1-\vstar^2}{\vstar^2}\right)\epsi^2 + O(\epsi^4)\right).&
		\end{aligned}
	\end{equation}
	Consequently, \added{when choosing $\vstar = 1$ (and also for $\hydroLimit$ when $\vstar\neq 1$),} the spectrum of $\Bcal$ can be written as,
	\begin{equation}
		\Sp{\Bcal}
			\subset
		\disk\left(-\dfrac{1}{\epsi}, \max_{\zeta}\sqrt{\alpha^2+\frac{\beta^2}{\vstar^2}}\right)
			\cup
		\{\lambda_1\}.
	\end{equation}
\end{thm}
\begin{proof}
Since we discretized in space, the Fourier mode $\zeta$ is a discrete variable, $\zeta_i=i2\pi\dx$, $i=1, ..., I$. The matrix product $\MF\PF$ and matrix $\DF(\zeta)$ in the Fourier domain become:
\begin{equation} \label{eq:MP_and_D}
	\MF\PF = \frac{1}{2}\left( \begin{array}{ccc} 1+1/\vstar & 1+1/\vstar \\ 1-1/\vstar & 1-1/\vstar \end{array} \right), \qquad 
	\DF = \left( \begin{array}{ccc} z^{+} & 0 \\ 0 & z^{-} \end{array} \right), 
\end{equation}
where we used the matrices $\MF$ and $\PF$ given in equations \eqref{eq:M_fourier} and \eqref{eq:P_fourier}. Furthermore, in \eqref{eq:MP_and_D} we adopted the shorthand notation $z^{\pm} = \alpha \pm \imath\beta$, where $\alpha$ and $\beta$ depend on the spatial discretization method, the velocity value $\vstar$ and the Fourier mode $\zeta$.

We start by calculating the spectrum of the matrix $\Acal = \epsi\Bcal$. The eigenvalues $\tilde\lambda$ of $\Acal$ can be obtained as the roots of its characteristic polynomial $\chi_{\Acal}(\tilde\lambda)$, which is given by,
\begin{equation} \label{eq:Acal_char_poly}
	\begin{aligned}
	\chi_{\Acal}(\tilde\lambda) &= \abs{\Acal - \tilde\lambda\eye}& \\
	&= \tilde\lambda^2 + (1 + \epsi(z^{+}+z^{-}))\tilde\lambda + \epsi\left(\frac{1+\vstar}{2\vstar}z^{+}-\frac{1-\vstar}{2\vstar}z^{-}\right) + \epsi^2z^{+}z^{-} = 0.&
	\end{aligned}
\end{equation}
Using the equalities, 
\[ z^{+}+z^{-} = 2\alpha, \qquad z^{+}-z^{-} = \imath2\beta, \qquad z^{+}z^{-} = \alpha^2+\beta^2, \]
the roots of equation \eqref{eq:Acal_char_poly} can be calculated as:
\begin{equation} \label{eq:Acal_eigs} 
	\tilde\lambda_{1,2} = \frac{-1-2\alpha\epsi \pm \sqrt{\discriminant}}{2},
\end{equation}
where the discriminant $\discriminant$ of equation \eqref{eq:Acal_char_poly} can be written as,
\[ \discriminant = 1 - 4\beta^2\epsi^2 - \imath4\frac{\beta}{\vstar}\epsi. \]
Employing a Taylor series expansion for $\sqrt{\discriminant}$ for $\epsi \rightarrow 0$ in equation \eqref{eq:Acal_eigs} leads to:
\begin{equation*} 
	\sqrt{\discriminant} = 1 + 2\beta^2\left(\frac{1 - \vstar^2}{\vstar^2}\right)\epsi^2 + O(\epsi^4) + \imath\left(-2\frac{\beta}{\vstar}\epsi + 4\frac{\beta^3}{\vstar}\left(\frac{1-\vstar^2}{\vstar^2}\right)\epsi^3 + O(\epsi^5)\right).
\end{equation*}
Plugging this expansion into equation \eqref{eq:Acal_eigs} we obtain,
\begin{align*} 
	\tilde\lambda_1 &= \frac{-1-2\alpha\epsi + \sqrt{\discriminant}}{2} \\
	& = -\alpha\epsi + \beta^2\left(\frac{1 - \vstar^2}{\vstar^2}\right)\epsi^2 + O(\epsi^4) + \imath\left(-\frac{\beta}{\vstar}\epsi + 2\frac{\beta^3}{\vstar}\left(\frac{1-\vstar^2}{\vstar^2}\right)\epsi^3 + O(\epsi^5)\right) \\
	\tilde\lambda_2 &= \frac{-1-2\alpha\epsi - \sqrt{\discriminant}}{2} \\
	& = -1 - \alpha\epsi - \beta^2\left(\frac{1 - \vstar^2}{\vstar^2}\right)\epsi^2 + O(\epsi^4) + \imath\left(\frac{\beta}{\vstar}\epsi - 2\frac{\beta^3}{\vstar}\left(\frac{1-\vstar^2}{\vstar^2}\right)\epsi^3 + O(\epsi^5)\right).
\end{align*}
Finally, the dominant and fast eigenvalues, $\lambda_1$ and $\lambda_2$, of the matrix $\Bcal = \dfrac{1}{\epsi}\Acal$ are then given by
\begin{align*}
	\lambda_1 &= -\alpha + \beta^2\left(\frac{1 - \vstar^2}{\vstar^2}\right)\epsi + O(\epsi^3) + \imath\left(-\frac{\beta}{\vstar} + 2\frac{\beta^3}{\vstar}\left(\frac{1-\vstar^2}{\vstar^2}\right)\epsi^2 + O(\epsi^4)\right) \\
	\lambda_2 &= -\frac{1}{\epsi} - \alpha - \beta^2\left(\frac{1 - \vstar^2}{\vstar^2}\right)\epsi + O(\epsi^3) + \imath\left(\frac{\beta}{\vstar} - 2\frac{\beta^3}{\vstar}\left(\frac{1-\vstar^2}{\vstar^2}\right)\epsi^2 + O(\epsi^4)\right).
\end{align*}
\end{proof}

When we write the Fourier transform of the inner forward Euler scheme \eqref{eq:FE_scheme} as 
 \begin{equation} \label{eq:timestepper_fourier}
	 \fF^{k+1} = \SdtF\fF^{k} = (\eye+\dt\Bcal)\fF^k,
\end{equation}
it is clear that the amplifications factors $\tauvec=(\tau_1,\tau_2)$ of the forward Euler scheme, which are the eigenvalues of $\SdtF$, and the eigenvalues $\lambdavec=(\lambda_1,\lambda_2)$ of the matrix $\Bcal$ are related via 
\begin{equation} \label{eq:tau_fe} 
	\tau_j = 1 + \dt \lambda_j, \qquad j \in \{1,2\}.
\end{equation}
Thus, the spectrum of an inner forward Euler time-stepper satisfies
\begin{equation}\label{eq:spec_inner_fe}
	\Sp{\eye + \dt\Bcal}
		\subset
	\disk\left(1-\dfrac{\dt}{\epsi}, \dt\max_{\zeta}\sqrt{\alpha^2+\frac{\beta^2}{\vstar^2}}\right)
		\cup
	\{1+\lambda_1\dt\},
\end{equation}
with $\lambda_1$ given in theorem \ref{thm:spectrum}.

For higher-order Runge-Kutta inner integrators (of order $Q$), we have 
\begin{equation} \label{eq:tau_rk} 
	\tau_j = 1 + \sum_{q=1}^Q\frac{(\dt\lambda_j)^q}{q!}, \qquad j \in \{1,2\}, 
\end{equation} 
and thus 
\begin{equation} \label{eq:spec_inner_rk}
	\Sp{\eye +\sum_{q=1}^Q\dfrac{(\dt \Bcal)^q}{q!}}
		\subset
	\disk\left(1 + \sum_{q=1}^Q\dfrac{(-1)^q}{q!}\left(\dfrac{\dt}{\epsi}\right)^q, C\dt\max_{\zeta}\sqrt{\alpha^2+\frac{\beta^2}{\vstar^2}}\right)
		\cup
	\left\{1 + \sum_{q=1}^Q\frac{(\dt \lambda_1)^q}{q!}\right\},
\end{equation}
in which the constant $C$ depends on the order of the Runge-Kutta method. The spectrum in formula \eqref{eq:spec_inner_rk} can be obtained by transforming the spectrum of $\Bcal$ in \eqref{eq:spectrum_B} by the particular expression of the Runge-Kutta inner integrator amplification factor given by \eqref{eq:tau_rk}.

\subsection{Method parameters} \label{sec:param}
In this section the projective integration method parameters will be determined by ensuring that the spectrum of the inner integrator falls within the stability region of the \changed{projective forward Euler method}. First, we select a suitable inner time step $\dt$ such that the fast modes are quickly damped (section \ref{sec:choose-dt}). Then, we choose the outer time step $\Dt$ commensurate with the slow part of the evolution (section \ref{sec:choose-Dt}). Finally, we fix $K$ to ensure overall stability (section~\ref{sec:choose-K}).
 
\subsubsection{Choice of inner integrator and time step} \label{sec:choose-dt}
Let us first discuss the effect of the choice of the inner integrator.  To this end, we look at the discretization error and the desired stability properties.  
 
Concerning stability, we deduce from the stability properties of the projective integration method (see section \ref{sec:stab_pi}) that it is preferable to center the part of the spectrum of the inner time-stepper corresponding to quickly damped modes around $0$.  Since, for forward Euler, these fast modes are given by \eqref{eq:spec_inner_fe}, we choose, for inner forward Euler, $\dt=\epsi$.
 
For higher-order inner integrators with even order, we immediately see that one cannot center eigenvalues corresponding to the quickly damped modes around $0$. For instance, for a second order Runge-Kutta method we have that 
\begin{equation} \label{eq:max_eig_dist} 
	\min_{\dt\lambda_2} \tau_2 = \min_{\dt\lambda_2}\left(1 + \dt\lambda_2 + \dfrac{1}{2}(\dt\lambda_2)^2\right) = \dfrac{1}{2}, 
\end{equation}
and this value is reached for $\dt=\epsi$. Moreover, as we will show below, the discretization error of the projective integration scheme is dominated by the error of the outer integrator, whereas the discretization error due to the inner integrator is negligible. As a consequence, we conclude that there is no point in using a higher-order time discretization for the inner integrator.
 
\subsubsection{Outer time step} \label{sec:choose-Dt}
Given the inner time step $\dt=\epsi$, we can choose $\Dt$ such that the dominant eigenvalue $\tau_1$ of the inner integrator lies inside the stability region $\diskprk{$q$}$. Let us first look at the projective forward Euler method, with stability regions \eqref{eq:stab_pfe}. We have the following condition on $\Dt$ such that $\tau_{1}$ is contained within $\diskpfe_1$,
\[ \sqrt{\left(\Re(\tau_1) -  \left(1-\frac{\epsi}{\Dt}\right)\right)^2 + \Big(\Im(\tau_1)\Big)^2} \le \frac{\epsi}{\Dt}. \]
Using \eqref{eq:tau_fe} and $\lambda_1$ given in theorem \ref{thm:spectrum}, we deduce a bound of the form
\begin{equation} \label{eq:Deltat_restriction}
	\Dt \le \min_\zeta\left(\frac{2\alpha \vstar^2}{\alpha^2\vstar^2 + \beta^2}\right).
\end{equation}
For the projective forward Euler method and first order upwind\added{, using table \ref{tab:upwind},} we obtain the CFL-like bound:
\[ \Dt \le \frac{\dx}{\vstar}. \]
\added{As indicated in section \ref{subsec:1d}, we need to choose $\vstar \ge \max_u \abs{F'(u)}$ to satisfy the subcharacteristic condition~\eqref{eq:subcharacteristic_condition}. When choosing $\vstar = \max_u \abs{F'(u)}$,} we obtain exactly the stability condition for the forward Euler method applied to the original hyperbolic conservation law \eqref{eq:cons_law_1d_scalar}. 

When we combine a higher-order upwind method with forward Euler time-stepping, the time step restriction becomes much more severe, in exactly the same way as would be the case with a direct higher-order upwind/forward Euler discretization of equation \eqref{eq:cons_law_1d_scalar}. \added{Using table \ref{tab:upwind} and equation \eqref{eq:Deltat_restriction}, we find:}
\begin{align*}
	\Dt &\le \frac{2\pi^2\vstar(3-10\pi^2\dx^2)}{3}\dx^3, &\text{(second order upwind)} \\
	\Dt &\le \frac{2\pi^2\vstar(3-2\pi^2\dx^2)}{9}\dx^3 &\text{(third order upwind)}.
\end{align*}
This is because, for higher-order upwind methods, the dominant eigenvalues do not lie on a circle anymore. In fact, they belong to a region that is much steeper close to 1. Since the dominant stability region of projective forward Euler is always a circle, we need to choose its radius sufficiently large such that even the steepest eigenvalues fall into this circular stability region.
When we use a higher-order projective Runge-Kutta method, we should determine the bound on $\Dt$ using its proper dominant stability region (see, e.g., \cite{Lejon2016}). For instance, in case of the second order projective Runge-Kutta method, we obtain the following CFL-like bound,
\[ \Dt \le \nu\dx, \]
where $\nu$ is the CFL-number which depends on the coefficients $\alpha$ and $\beta$. Notice that in both cases $\Dt$ never depends on $\epsi$.
 
\subsubsection{The number $K$ of inner time steps} \label{sec:choose-K}
The only parameter that remains to be chosen is the number $K$ of inner steps such that all fast eigenvalues $\tau_2$ (see equation \eqref{eq:tau_fe} or \eqref{eq:tau_rk}) corresponding to quickly damped modes are contained within the stability region $\diskpfe_2$ (see equation \eqref{eq:stab_pfe}).  Given $\dt$, $\Dt$ and the radius of the fast eigenvalues zone introduced in \eqref{eq:spec_inner_fe} or \eqref{eq:spec_inner_rk}, we obtain a condition for $K$ of the form,
\[ c\epsi \le \left(\dfrac{\epsi}{\Dt}\right)^{1/K},  \qquad c = \max_{\zeta}\sqrt{\alpha^2+\frac{\beta^2}{\vstar^2}}. \] 
From this we can obtain the bound
\[ K \ge \frac{1}{1 + \dfrac{\log(c)}{\log(\epsi)}} + \frac{\log(\Dt)}{\log\left(\dfrac{1}{\epsi c}\right)}. \]
Therefore, for $\epsi$ small, a safe choice is taking $K \ge 2$, uniformly in $\epsi$. The computational cost of the method is then independent of $\epsi$.

\changed{Note that, when choosing a higher-order inner Runge-Kutta method, the fast eigenvalues $\tau_2$ cannot be centered around $0$, see equation \eqref{eq:max_eig_dist}. } In that case, we need to impose that the radius of the stability region $\diskpfe_2$ (see equation \eqref{eq:stab_pfe}) is sufficiently large to encompass the fast eigenvalues zone of the inner Runge-Kutta method, yielding
\[ \tilde c \le \left(\dfrac{\dt}{\Dt}\right)^{1/K}, \qquad \tilde c = 1 + \sum_{q=1}^Q\dfrac{(-1)^q}{q!}\left(\dfrac{\dt}{\epsi}\right)^q + C\dt\max_{\zeta}\sqrt{\alpha^2+\frac{\beta^2}{\vstar^2}}. \] 
Taking the logarithm of both sides and rearranging terms ultimately leads to
\[ K \ge \frac{\log\left(\dfrac{1}{\dt}\right)}{\log\left(\dfrac{1}{\tilde c}\right)} + \frac{\log(\Dt)}{\log\left(\dfrac{1}{\tilde c}\right)}, \]
which results in a condition of the form $K\ge\log(1/\epsi)$ since $\dt = O(\epsi)$. Hence, using an inner integrator of \changed{higher} order destroys the asymptotic-preserving nature of the projective integration method.

\subsection{Consistency analysis} \label{subsec:consist}
\added{We now examine the consistency behavior of the proposed method, that is, the behavior of the local truncation error. The exposition in this section is based on the consistency derivation in \cite{Lafitte2012,Lejon2016}, which becomes simpler in the hyperbolic case. We introduce the following notation:
\begin{enumerate}
	\item $\fSDSys^N$ and $\fExact^N$ denote the numerical and exact solution at time $t^N = N\Dt$, respectively. Both are vectors of length $I \times J$ obtained by collecting $f_{i,j}^N$ and $f(x_i,v_j,t^N)$, $\forall\,1 \le i \le I,\, 1 \le j \le J$;
	\item $\uSDSys^N = \averageV{\fSDSys^N}$ and $\uExact^N = \averageV{\fExact^N}$ represent the numerical and exact conserved quantity at time $t^N$ obtained by averaging over velocity space. Both are vectors of length $I$.
\end{enumerate}
Since the goal of the simulations is to obtain $\uSys(x,t) = \averageV{\f^\epsi(x,v,t)}$, we define the local truncation error at time $t=t^{N+1}$ as:
\begin{equation} \label{eq:local_truncation_error}
	E^{N+1} = \frac{\uExact^{N+1} - \uSDSys^{N+1}}{\Dt}.
\end{equation}

We focus here on the consistency analysis for PFE. The analysis can straightforwardly be extended to higher order PRK methods, see \cite{Lejon2016}, since the stages $\mathbf{k}_s$ in the projective Runge-Kutta methods are computed as finite difference approximations of the time derivative in which the function values are obtained by the PFE method, see equations \eqref{eq:PRK_stage_1}-\eqref{eq:PRK_stage_s}. 
By averaging equation \eqref{eq:PFE_scheme} over velocity space and substituting the result in equation \eqref{eq:local_truncation_error}, we find:
\begin{equation} \label{eq:local_truncation_error_PFE}
	E^{N+1} = \frac{\uExact^{N+1} - \uSDSys^{N,K+1}}{\Dt} - \left(\frac{\Dt - (K+1)\dt}{\Dt}\right)\frac{\uSDSys^{N,K+1} - \uSDSys^{N,K}}{\dt}.
\end{equation}
Equation \eqref{eq:local_truncation_error_PFE} shows that the local truncation error of the PFE method depends on its inner integrator. Therefore, we also introduce the local truncation error of the inner integrator, defined as:
\begin{equation} \label{eq:local_truncation_error_inner}
	e_f^{N,k+1} = \frac{\fExact^{N,k+1} - \fSDSys^{N,k+1}}{\dt}.
\end{equation}
We now calculate an estimate for $e_f^{N,k+1}$. We rewrite both quantities $\fExact^{N,k+1}$ and $\fSDSys^{N,k+1}$ in terms of the the solutions at time $t^{N,k}$. The numerical solution $\fSDSys^{N,k+1}$ is reformulated by recalling that we choose the FE scheme \eqref{eq:FE_scheme} as inner integrator (see section \ref{sec:choose-dt}), and the fact that we require $\dt = \epsi$ for stability (see section \ref{sec:choose-dt}), giving:
\begin{align} \label{eq:local_truncation_error_inner_first_term}
	\fSDSys^{N,k+1} &= \Sdt(\fSDSys^{N,k}) = \fSDSys^{N,k} +  \epsi\left(-\Dxv(\fSDSys^{N,k}) + \frac{1}{\epsi}\left(\Mv(\uSys^{N,k}) - \fSDSys^{N,k}\right)\right) \notag \\
	&= -\epsi\Dxv(\fSDSys^{N,k}) + \Mv(\uSys^{N,k}).
\end{align}
The exact solution $\fExact^{N,k+1}$ is computed by applying a Taylor series expansion of $\fExact^{N,k+1}$ around time $t^{N,k}$, yielding:
\begin{align} \label{eq:local_truncation_error_inner_second_term}
	\fExact^{N,k+1} &= \fExact^{N,k} + \epsi\,\partial_t \fExact^{N,k} + O(\epsi^2) \notag \\
	&= \fExact^{N,k} + \epsi\left(-v\partial_x \fExact^{N,k} + \frac{1}{\epsi}\left(\Mv(\uExact^{N,k}) - \fExact^{N,k}\right)\right) + O(\epsi^2) \notag \\
	&= \Sdt(\fExact^{N,k}) + \epsi\left(\Dxv(\fExact^{N,k}) - v\partial_x \fExact^{N,k}\right) + O(\epsi^2),
\end{align}
where the last equality is obtained by adding and subtracting $\Dxv(\fExact^{N,k})$ and subsequently using the forward Euler timestepper expression \eqref{eq:FE_scheme}. Substituting equations \eqref{eq:local_truncation_error_inner_first_term} and \eqref{eq:local_truncation_error_inner_second_term} into \eqref{eq:local_truncation_error_inner}, we find:
\begin{equation} \label{eq:local_truncation_error_inner_recursive}
	e_f^{N,k+1} = \Sdt(e_f^{N,k}) + \left(\Dxv(\fExact^{N,k}) - v\partial_x \fExact^{N,k}\right) + O(\epsi).
\end{equation}
Working out the recursion in equation \eqref{eq:local_truncation_error_inner_recursive} gives:
\begin{equation} \label{eq:local_truncation_error_inner_norecursion}
	e_f^{N,K+1} = \sum_{k=0}^{K}\Sdt^k\left(\Dxv(\fExact^{N,K-k}) - v\partial_x \fExact^{N,K-k}\right) + O((K+1)\epsi).
\end{equation}
Since the difference between brackets in equation \eqref{eq:local_truncation_error_inner_norecursion} precisely corresponds to the spatial discretization error, we find the following estimate for the inner integrator local truncation error:
\begin{equation} \label{eq:local_truncation_error_inner_final}
	e_f^{N,K+1} = O((K+1)\dx^p) + O((K+1)\epsi),
\end{equation}
where $p$ denotes the order of accuracy of the spatial discretization method.
Then, using equations \eqref{eq:local_truncation_error_inner} and \eqref{eq:local_truncation_error_inner_final}, the local truncation error \eqref{eq:local_truncation_error_PFE} of the PFE method is computed as:
\begin{align} \label{eq:local_truncation_error_PFE_calculate}
	E^{N+1} &= \frac{\uExact^{N+1} - \uExact^{N,K+1} + \epsi\averageV{e_f^{N,K+1}}}{\Dt} - \left(\frac{\Dt - (K+1)\epsi}{\Dt}\right)\frac{\uExact^{N,K+1} - \uExact^{N,K} - \epsi\averageV{e_f^{N,K+1} - e_f^{N,K}}}{\epsi} \notag \\
	&= O(\Dt) + O\left(\frac{\epsi^2}{\Dt}\right) + \frac{\epsi}{\Dt}\averageV{e_f^{N,K+1}} + \left(\frac{\Dt - (K+1)\epsi}{\Dt}\right)\Big(O(\dx^p) + O((K+1)\epsi)\Big) \notag \\
	&= O(\Dt) + O(\dx^p) + O((K+1)\epsi) + O\left(\frac{\epsi^2}{\Dt}\right) + \frac{\epsi}{\Dt}O(\dx^p),
\end{align}
where we used Taylor expansions of all quantities $\uExact$ around $t^N$ in the first equality.

Ultimately, we obtain the following expression for the local truncation error for PFE and FE as inner integrator:
\begin{equation} \label{eq:local_truncation_error_PFE_final}
	E^{N+1} = O(\Dt) + O(\dx^p) + O((K+1)\epsi) + \epsi O\left(\frac{\dx^p + \epsi}{\Dt}\right),
\end{equation}
Under the above assumptions, it can then easily be shown, following the proof of~\cite[Theorem~5.1]{Lejon2016}, that a projective Runge-Kutta method of order $q$ has the following discretization error:
\begin{equation} \label{eq:local_trunc_error} 
	E^N = O\left(\Dt^q\right) + O\left(\dx^p\right) + O\left((K+1)\epsi\right) + \epsi O\left(\frac{\dx^p + \epsi}{\Dt}\right).
\end{equation}}
The first term in \eqref{eq:local_trunc_error} is due to the time discretization error made in the outer Runge-Kutta integrator, whereas the next two terms are the space and time discretization error respectively due to the inner integrator. The last term results from the time derivative operator approximation. We remark that, in the limit of $\Dt$ going to $0$, for fixed $\epsi$, this last term would result in divergence. However, since the goal is to create an asymptotic-preserving scheme, valid for fixed $\Dt$ (independent of $\epsi$), while $\epsi$ tends to $0$, this is not an issue: the last term then becomes of $O(\epsi)$, and hence negligible.

\added{Finally, we mention that the choice of $\epsi$ in this work is determined by the finite difference approximation to the time derivative appearing in the projective Runge-Kutta formulations which have the following form:}
\begin{equation} \label{eq:FD_time_derivative}
	\added{\frac{f^{n,K+1} - f^{n,K}}{\epsi}.}
\end{equation}
\added{As a consequence, the numerical error of this approximation is $O(\epsi)$ and is bounded below by $O(\sqrt{\epsi_{mach}})$ with $\epsi_{mach} \approx 10^{-16}$ being the machine precision. For that reason, we choose $\epsi=10^{-8}$.}


\section{Applications} \label{sec:results}
Let us now illustrate the relaxation method with projective integration on a number of example systems. We first examine the one-dimensional case. In section \ref{subsec:order}, we consider the linear advection equation, and demonstrate the spatial and temporal order of the methods.  Subsequently, we investigate nonlinear conservation laws, Burgers' equation in section \ref{subsec:burgers} and the Euler equations (Sod's shock test) in section \ref{subsec:sod}. Afterwards, we consider linear advection, the dam-break problem and the Euler equations in the two-dimensional case (sections \ref{subsec:lin_2d}--\ref{subsec:euler_2d}). 

\subsection{Linear advection in 1D} \label{subsec:order}
Let us first illustrate the order of the relaxation method with projective integration, both in space and time. To this end, we consider the linear advection equation, i.e., equation \eqref{eq:cons_law_1d_scalar} with the linear flux function \eqref{eq:flux_linear_advection_1d_scalar}, 
\begin{equation} \label{eq:linear_advection_1d_scalar}
	\partial_t u + \partial_x (a\cdot u)= 0,
\end{equation}
in which the macroscopic unknown function $u(x,t)$ denotes the density of particles. We compute the solution for $t\in[0,T]$ and $x\in[0,1]$, using $a = 1$. We impose periodic boundary conditions and choose a smooth initial condition:
\begin{equation} \label{eq:linear_advection_1d_init}
	u(x,0) = \exp(-100(x-0.5)^2).
\end{equation}
\changed{We compute the global error $E^N$ at time $t^N = T$ which is defined as $E^N=\norm{\boldsymbol{E}^N}$, with $\boldsymbol{E}^N=\set{E_i^N}{i}$, and $E_i^N$ denoting the global space-time discretization error at time $t=t^N$ and grid location $x=x_i$ given by:
\begin{equation} \label{eq:global_error}
	E^N_i = \abs{\uSys_i^N-\uSys(x_i,t^N)},
\end{equation}
(For the linear advection equation, the exact solution is known analytically.) Here, we always choose the $1$-norm to calculate the global error $E^N$.}

For the relaxation method, we use the kinetic equation \eqref{eq:kin_eq_1d}, in which we discretize the velocity space using $J=2$ velocities. Taking into account the subcharacteristic condition \eqref{eq:subcharacteristic_condition_artificial}, the velocities are chosen as: $v_1 = -\abs{a}$ and $v_2 = \abs{a}$. We point out that, in this purely academic first test case, we recover the Jin-Xin relaxation system (see also Remark 2.1), and additionally, the distribution corresponding to $v_1=-\abs{a}$ vanishes in the limit of $\epsi$ tending to zero. 
As the Maxwellian, we choose \eqref{eq:maxwellian_artificial_1d}, with $F(u) = a\cdot u$. The inner integrator is a space-time discretization of equation \eqref{eq:kin_eq_1d}, in which we choose the standard upwind spatial discretizations of order $1$, $2$ and $3$ with grid spacing $\dx$ (that will vary throughout the experiments), combined with a forward Euler time discretization with $\dt=\epsi$ and $\epsi=10^{-8}$. The projective integration method uses $K=2$ inner steps, and an outer time step of size $\Dt$ (that will also vary). 

\paragraph{Numerical spatial order (figures \ref{fig:order_test_PFE_PRK2} and \ref{fig:order_test_PRK4})}
To illustrate the spatial order of accuracy, we calculate the error on time $T=0.02$ and vary the grid spacing $\dx$ as 
\begin{equation} \label{eq:order_test_dx} 
	\dx = [0.04; 0.02; 0.01; 0.005; 0.002; 0.001; 0.0005; 0.0002; 0.0001],
\end{equation}
and correspondingly choose the outer time step $\Dt$ as
\begin{equation} \label{eq:order_test_Dt}
	\Dt = O\left(\dx^{p/q}\right) = C_p\dx^{p/q},
\end{equation}
\added{such that the temporal discretization error and the spatial discretization error display the same asymptotic behaviour as $\Delta x$ tends to zero, see equation \eqref{eq:local_trunc_error}.} 
The constants $C_p$ should be chosen such that the projective integration method remains stable for all choices of $\dx$. Figure \ref{fig:order_test_PFE_PRK2} shows the error as a function of $\dx$. In the leftmost figure, time integration is done using the projective forward Euler method (PFE), for which the time order $q=1$. The constants in \eqref{eq:order_test_Dt} are then chosen as $C_1 = 0.5$, $C_2 = 20$ and $C_3 = 100$. We clearly observe the expected spatial order. This is confirmed by fitting a least squares line through the calculated error points. The slopes of these lines correspond to the numerical order which in this case were found to be 0.98, 1.94 and 2.99. These indeed lie sufficiently close to the expected spatial order. In the middle figure, the experiment is repeated using a second order projective Runge-Kutta method (PRK2, $q = 2$). For spatial orders $p \in \{1,2\}$, we choose $\Dt = 0.5\dx$. In that case, the first term in expression \eqref{eq:local_trunc_error} will be dominant and the order in space can be observed. For $p = 3$ we put $\Dt = C_3\dx^{3/2}$ and choose $C_3 = 4$. On these plots, we observe that, for the third order upwind discretization, the error curves start to level off for small values of $\dx$. This is due to the fact that the contribution of the spatial discretization error in \eqref{eq:local_trunc_error} becomes negligible, and the $O(\epsi)$ term becomes dominant. \changed{As indicated in section \ref{subsec:consist}, this term results from the time derivative approximation in the projective step by a finite difference expression, see equation \eqref{eq:FD_time_derivative}.} When calculating the slopes of the least squares fit we now obtain 0.99, 1.99 and 3.07 which correspond to the expected spatial orders.

Finally, we repeat the experiment using a second order Runge-Kutta method as the inner integrator (see figure \ref{fig:order_test_PFE_PRK2}, right).  We note, in agreement with the observations in section \ref{sec:param} that $K$ can now no longer be chosen independently of $\epsi$.  Here, we choose $K=21$ for $p\in\{1,2\}$ and $K=22$ for $p=3$. The numerical orders are 1.00, 1.99 and 3.04. Moreover, we again observe that the error levels off, since the remaining error is due to the finite difference approximation of the time derivative, and not due to the time discretization error of the inner integrator. Note that the error curve now levels off to a value which is about $10$ times higher than in the forward Euler case, due to a less efficient damping of the fast eigenvalues, as can be seen in \eqref{eq:spec_inner_rk}. This observation is supported by looking at equation \eqref{eq:local_trunc_error} in which the value of $K$ needs to be taken $10$ times higher for RK2 as inner integrator so as to guarantee a stable functioning of the method. From these findings we conclude that it is not useful to select a higher-order inner integrator within the projective integration framework. Therefore, in what follows we will always select FE as inner integrator.

Next, we repeat this experiment using a fourth order projective Runge-Kutta method (PRK4). We choose $\Dt = C\dx$ for each $p \in \{1, 2, 3\}$ and put $C$ equal to $0.4$. The result is depicted on the left hand side plot of figure \ref{fig:order_test_PRK4}, where we again see the expected behavior (numerical orders: 0.99, 1.99 and 2.97).  To avoid unphysical oscillations associated with higher-order upwind schemes, we also performed the same experiments using an essentially non-oscillatory (ENO) spatial discretization \cite{eno_weno}, which uses an adaptive stencil thus trying to avoid stencils with large variations in the solution values. The order test of the PRK4 scheme with ENO is shown on the right hand side plot of figure \ref{fig:order_test_PRK4}. The calculated numerical orders were 0.99, 1.86 and 2.86 which are in agreement with the expected spatial orders of the ENO scheme.

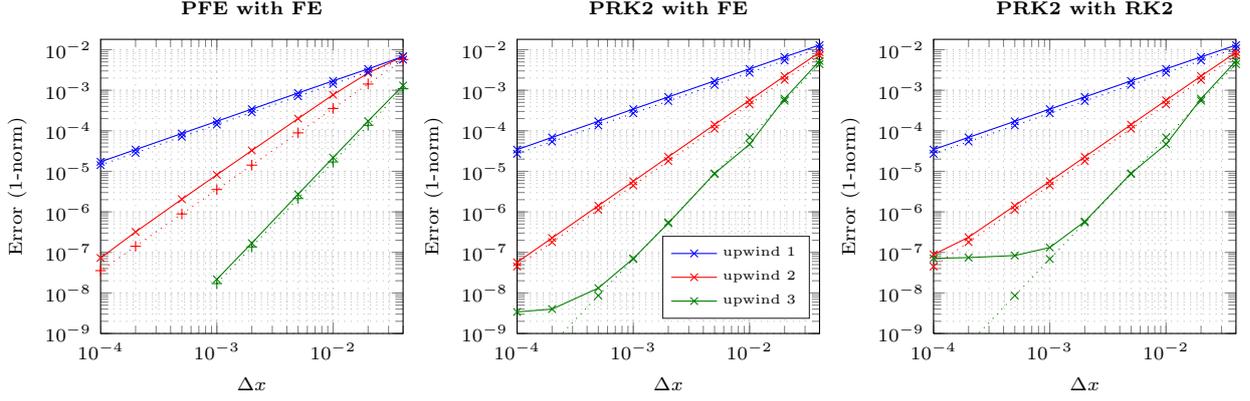
\begin{figure}
	\begin{center}
		\figname{ordertest_space_PFE_PRK2_FE_RK2}
%
%
\definecolor{mycolor1}{rgb}{0.00000,0.49804,0.00000}%
\newcommand{\figWidth}{0.25\textwidth} 
\newcommand{\figHeight}{3.9cm} 
\newcommand{\figSpacingRight}{1.5cm} 
\begin{tikzpicture}

\begin{axis}[%
width=0,
height=0,
scale only axis,
clip=false,
xmin=0,
xmax=1,
ymin=0,
ymax=1,
hide axis,
]
\node[align=center, text=black]
at (3*\figWidth/2+\figSpacingRight,\figHeight+0.9cm) {Order test (space) for linear advection};
\end{axis}

\begin{axis}[%
width=\figWidth,
height=\figHeight,
at={(0,0.0)},
scale only axis,
xmode=log,
xmin=0.0001,
xmax=0.04,
xminorticks=true,
xlabel={$\dx$},
xmajorgrids,
xminorgrids,
ymode=log,
ymin=1e-09,
ymax=0.018,
yminorticks=true,
ytick = {1e-9, 1e-8, 1e-7, 1e-6, 1e-5, 1e-4, 1e-3, 1e-2, 1e-1},
ylabel={Error (1-norm)},
ymajorgrids,
yminorgrids,
grid style={dotted},
axis background/.style={fill=white},
title style={font=\scriptsize\bfseries},
title={PFE with FE}
]
\addplot [color=red,solid,mark=x,mark options={solid},forget plot]
  table[]{tikz/data/ordertest_space_PFE_PRK2-13.tsv};
\addplot [color=mycolor1,solid,mark=x,mark options={solid},forget plot]
  table[]{tikz/data/ordertest_space_PFE_PRK2-14.tsv};
\addplot [color=blue,solid,mark=x,mark options={solid},forget plot]
  table[]{tikz/data/ordertest_space_PFE_PRK2-15.tsv};
\addplot [color=blue,dotted,mark=x,mark options={solid},forget plot]
  table[]{tikz/data/ordertest_space_PFE_PRK2-16.tsv};
\addplot [color=red,dotted,mark=+,mark options={solid},forget plot]
  table[]{tikz/data/ordertest_space_PFE_PRK2-17.tsv};
\addplot [color=mycolor1,dotted,mark=+,mark options={solid},forget plot]
  table[]{tikz/data/ordertest_space_PFE_PRK2-18.tsv};
\end{axis}

\begin{axis}[%
width=\figWidth,
height=\figHeight,
at={(\figWidth+\figSpacingRight,0)},
scale only axis,
xmode=log,
xmin=0.0001,
xmax=0.04,
xminorticks=true,
xlabel={$\dx$},
xmajorgrids,
xminorgrids,
ymode=log,
ymin=1e-09,
ymax=0.018,
yminorticks=true,
ytick = {1e-9, 1e-8, 1e-7, 1e-6, 1e-5, 1e-4, 1e-3, 1e-2, 1e-1},
ylabel={Error (1-norm)},
ymajorgrids,
yminorgrids,
grid style={dotted},
axis background/.style={fill=white},
title style={font=\scriptsize\bfseries},
title={PRK2 with FE},
legend style={at={(0.48,0.05)},anchor=south west,legend cell align=left,align=left,draw=white!15!black}
]
\addplot [color=blue,solid,mark=x,mark options={solid}]
  table[]{tikz/data/ordertest_space_PFE_PRK2-3.tsv};
\addlegendentry{upwind 1};
\addplot [color=blue,dotted,mark=x,mark options={solid},forget plot]
  table[]{tikz/data/ordertest_space_PFE_PRK2-4.tsv};
\addplot [color=red,solid,mark=x,mark options={solid}]
  table[]{tikz/data/ordertest_space_PFE_PRK2-5.tsv};
\addlegendentry{upwind 2};
\addplot [color=red,dotted,mark=x,mark options={solid},forget plot]
  table[]{tikz/data/ordertest_space_PFE_PRK2-6.tsv};
\addplot [color=mycolor1,solid,mark=x,mark options={solid}]
  table[]{tikz/data/ordertest_space_PFE_PRK2-1.tsv};
\addlegendentry{upwind 3};
\addplot [color=mycolor1,dotted,mark=x,mark options={solid},forget plot]
  table[]{tikz/data/ordertest_space_PFE_PRK2-2.tsv};
\end{axis}

\begin{axis}[%
width=\figWidth,
height=\figHeight,
at={(2*\figWidth+2*\figSpacingRight,0)},
scale only axis,
xmode=log,
xmin=0.0001,
xmax=0.04,
xminorticks=true,
xlabel={$\dx$},
xmajorgrids,
xminorgrids,
ymode=log,
ymin=1e-09,
ymax=0.018,
yminorticks=true,
ytick = {1e-9, 1e-8, 1e-7, 1e-6, 1e-5, 1e-4, 1e-3, 1e-2, 1e-1},
ylabel={Error (1-norm)},
ymajorgrids,
yminorgrids,
grid style={dotted},
axis background/.style={fill=white},
title style={font=\scriptsize\bfseries},
title={PRK2 with RK2}
]
\addplot [color=blue,solid,mark=x,mark options={solid},forget plot]
  table[]{tikz/data/ordertest_space_PFE_PRK2-7.tsv};
\addplot [color=blue,dotted,mark=x,mark options={solid},forget plot]
  table[]{tikz/data/ordertest_space_PFE_PRK2-8.tsv};
\addplot [color=red,solid,mark=x,mark options={solid},forget plot]
  table[]{tikz/data/ordertest_space_PFE_PRK2-9.tsv};
\addplot [color=red,dotted,mark=x,mark options={solid},forget plot]
  table[]{tikz/data/ordertest_space_PFE_PRK2-10.tsv};
\addplot [color=mycolor1,solid,mark=x,mark options={solid},forget plot]
  table[]{tikz/data/ordertest_space_PFE_PRK2-11.tsv};
\addplot [color=mycolor1,dotted,mark=x,mark options={solid},forget plot]
  table[]{tikz/data/ordertest_space_PFE_PRK2-12.tsv};
\end{axis}
\end{tikzpicture}%
	\end{center}
	\vspace{-0.4cm}\caption{\label{fig:order_test_PFE_PRK2} Spatial order test for PFE with FE as inner integrator (left) and PRK2 with FE (middle) and RK2 (right) as inner integrators and three different spatial orders. The error is computed using the 1-norm. On each plot, the solid lines represent the calculated error whereas the dotted line shows the expected error. }
\end{figure}

\begin{figure}
	\begin{center}
		\figname{ordertest_space_PRK4_FE}
%
%
\definecolor{mycolor1}{rgb}{0.00000,0.49804,0.00000}%
\newcommand{\figWidth}{0.25\textwidth} 
\newcommand{\figHeight}{4cm} 
\newcommand{\figSpacingRight}{2cm} 
\begin{tikzpicture}

\begin{axis}[%
width=0,
height=0,
scale only axis,
clip=false,
xmin=0,
xmax=1,
ymin=0,
ymax=1,
hide axis
]
\node[align=center, text=black] at (\figWidth+\figSpacingRight/2,\figHeight+0.9cm) {Order test (space) for linear advection: PRK4 + FE};
\end{axis}

\begin{axis}[%
width=\figWidth,
height=\figHeight,
at={(0,0.0)},
scale only axis,
xmode=log,
xmin=0.0001,
xmax=0.04,
xminorticks=true,
xlabel={$\dx$},
xmajorgrids,
xminorgrids,
ymode=log,
ymin=1e-09,
ymax=0.018,
yminorticks=true,
ytick = {1e-9, 1e-8, 1e-7, 1e-6, 1e-5, 1e-4, 1e-3, 1e-2, 1e-1},
ylabel={Error (1-norm)},
ymajorgrids,
yminorgrids,
grid style={dotted},
axis background/.style={fill=white},
title style={font=\scriptsize\bfseries},
title={upwind scheme},
legend style={at={(0.52,0.05)},anchor=south west,legend cell align=left,align=left,draw=white!15!black}
]
\addplot [color=blue,solid,mark=x,mark options={solid}]
  table[]{tikz/data/ordertest_space_PRK4-1.tsv};
\addlegendentry{upwind 1};
\addplot [color=blue,dotted,mark=x,mark options={solid},forget plot]
  table[]{tikz/data/ordertest_space_PRK4-2.tsv};
\addplot [color=red,solid,mark=x,mark options={solid}]
  table[]{tikz/data/ordertest_space_PRK4-3.tsv};
\addlegendentry{upwind 2};
\addplot [color=red,dotted,mark=x,mark options={solid},forget plot]
  table[]{tikz/data/ordertest_space_PRK4-4.tsv};
\addplot [color=mycolor1,solid,mark=x,mark options={solid}]
  table[]{tikz/data/ordertest_space_PRK4-5.tsv};
\addlegendentry{upwind 3};
\addplot [color=mycolor1,dotted,mark=x,mark options={solid},forget plot]
  table[]{tikz/data/ordertest_space_PRK4-6.tsv};
\end{axis}

\begin{axis}[%
width=\figWidth,
height=\figHeight,
at={(\figWidth+\figSpacingRight,0)},
scale only axis,
xmode=log,
xmin=0.0001,
xmax=0.04,
xminorticks=true,
xlabel={$\dx$},
xmajorgrids,
xminorgrids,
ymode=log,
ymin=1e-09,
ymax=0.018,
yminorticks=true,
ytick = {1e-9, 1e-8, 1e-7, 1e-6, 1e-5, 1e-4, 1e-3, 1e-2, 1e-1},
ylabel={Error (1-norm)},
ymajorgrids,
yminorgrids,
grid style={dotted},
axis background/.style={fill=white},
title style={font=\scriptsize\bfseries},
title={ENO scheme},
legend style={at={(0.55,0.05)},anchor=south west,legend cell align=left,align=left,draw=black}
]
\addplot [color=blue,solid,mark=x,mark options={solid}]
  table[]{tikz/data/ordertest_space_PRK4-7.tsv};
\addlegendentry{ENO 1};
\addplot [color=blue,dotted,mark=x,mark options={solid},forget plot]
  table[]{tikz/data/ordertest_space_PRK4-8.tsv};
\addplot [color=red,solid,mark=x,mark options={solid}]
  table[]{tikz/data/ordertest_space_PRK4-9.tsv};
\addlegendentry{ENO 2};
\addplot [color=red,dotted,mark=x,mark options={solid},forget plot]
  table[]{tikz/data/ordertest_space_PRK4-10.tsv};
\addplot [color=mycolor1,solid,mark=x,mark options={solid}]
  table[]{tikz/data/ordertest_space_PRK4-11.tsv};
\addlegendentry{ENO 3};
\addplot [color=mycolor1,dotted,mark=x,mark options={solid},forget plot]
  table[]{tikz/data/ordertest_space_PRK4-12.tsv};
\end{axis}
\end{tikzpicture}%
	\end{center}
	\vspace{-0.4cm}\caption{\label{fig:order_test_PRK4} Spatial order test for PRK4 with FE as inner integrator and three different spatial orders using a) upwind differences (left plot) and b) the ENO scheme (right plot). The error is computed using the 1-norm. On each plot, the solid lines represent the calculated error whereas the dotted line shows the expected error. }
\end{figure}
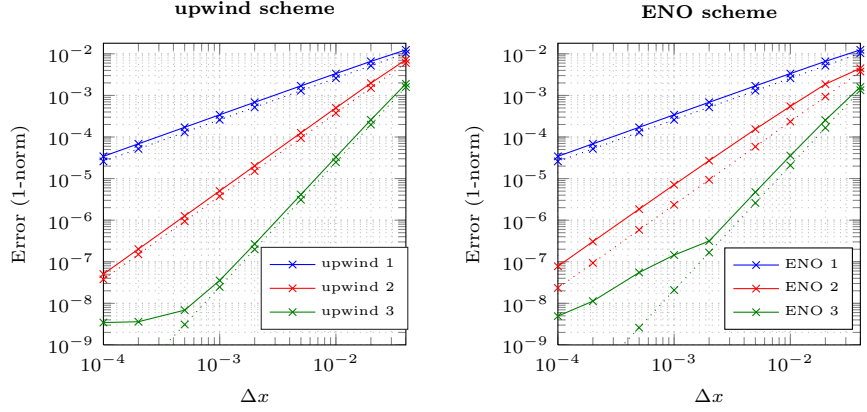

\begin{figure}
	\begin{center}
		\figname{ordertest_time_lin_adv}
%
%
\definecolor{mycolor1}{rgb}{0.00000,0.49804,0.00000}%
\newcommand{\figWidth}{0.25\textwidth} 
\newcommand{\figHeight}{4cm} 
\newcommand{\figSpacingRight}{2cm} 
\begin{tikzpicture}
\begin{axis}[%
width=0,
height=0,
scale only axis,
clip=false,
xmin=0,
xmax=1,
ymin=0,
ymax=1,
hide axis
]
\node[align=center, text=black] at (\figWidth+\figSpacingRight/2,\figHeight+0.9cm) {Order test (time) for linear advection};
\end{axis}

\begin{axis}[%
width=\figWidth,
height=\figHeight,
at={(0,0.0)},
scale only axis,
xmode=log,
xmin=0.0001,
xmax=0.04,
xminorticks=true,
xlabel={$\Dt$},
xmajorgrids,
xminorgrids,
ymode=log,
ymin=1e-09,
ymax=0.04,
yminorticks=true,
ytick = {1e-9, 1e-8, 1e-7, 1e-6, 1e-5, 1e-4, 1e-3, 1e-2, 1e-1},
ylabel={Error (1-norm)},
ymajorgrids,
yminorgrids,
grid style={dotted},
axis background/.style={fill=white},
title style={font=\scriptsize},
title={$\boldsymbol{\epsi = 10^{-5}}$},
legend style={at={(0.55,0.05)},anchor=south west,legend cell align=left,align=left,draw=white!15!black}
]
\addplot [color=blue,solid,mark=o,mark options={solid}]
  table[]{tikz/data/ordertest_time_linadv-1.tsv};
\addlegendentry{PFE}
\addplot [color=blue,dotted,mark=o,mark options={solid},forget plot]
  table[]{tikz/data/ordertest_time_linadv-2.tsv};
\addplot [color=red,solid,mark=o,mark options={solid}]
  table[]{tikz/data/ordertest_time_linadv-3.tsv};
\addlegendentry{PRK2}
\addplot [color=red,dotted,mark=o,mark options={solid},forget plot]
  table[]{tikz/data/ordertest_time_linadv-4.tsv};
\addplot [color=mycolor1,solid,mark=o,mark options={solid}]
  table[]{tikz/data/ordertest_time_linadv-5.tsv};
\addlegendentry{PRK4}
\addplot [color=mycolor1,dotted,mark=o,mark options={solid},forget plot]
  table[]{tikz/data/ordertest_time_linadv-6.tsv};
\end{axis}

\begin{axis}[%
width=\figWidth,
height=\figHeight,
at={(\figWidth+\figSpacingRight,0)},
scale only axis,
xmode=log,
xmin=0.0001,
xmax=0.04,
xminorticks=true,
xlabel={$\Dt$},
xmajorgrids,
xminorgrids,
ymode=log,
ymin=1e-09,
ymax=0.04,
yminorticks=true,
ytick = {1e-9, 1e-8, 1e-7, 1e-6, 1e-5, 1e-4, 1e-3, 1e-2, 1e-1},
ylabel={Error (1-norm)},
ymajorgrids,
yminorgrids,
grid style={dotted},
axis background/.style={fill=white},
title style={font=\scriptsize},
title={$\boldsymbol{\epsi = 10^{-8}}$}
]
\addplot [color=blue,solid,mark=o,mark options={solid},forget plot]
  table[]{tikz/data/ordertest_time_linadv-7.tsv};
\addplot [color=blue,dotted,mark=o,mark options={solid},forget plot]
  table[]{tikz/data/ordertest_time_linadv-8.tsv};
\addplot [color=red,dotted,mark=o,mark options={solid},forget plot]
  table[]{tikz/data/ordertest_time_linadv-10.tsv};
\addplot [color=red,solid,mark=o,mark options={solid},forget plot]
  table[]{tikz/data/ordertest_time_linadv-9.tsv};
\addplot [color=mycolor1,solid,mark=o,mark options={solid},forget plot]
  table[]{tikz/data/ordertest_time_linadv-11.tsv};
\addplot [color=mycolor1,dotted,mark=o,mark options={solid},forget plot]
  table[]{tikz/data/ordertest_time_linadv-12.tsv};
\end{axis}
\end{tikzpicture}%
	\end{center}
	\vspace{-0.4cm}\caption{\label{fig:order_test_time_lin_adv} Temporal order test for the different projective integration methods under study: PFE, PRK2 and PRK4 with FE as inner integrator and upwind differences of order 3 in space comparing results for $\epsi=10^{-5}$ (left plot) and $\epsi=10^{-8}$ (right plot). The error is computed using the 1-norm. On each plot, the solid lines represent the calculated error whereas the dotted line shows the expected error. }
\end{figure}
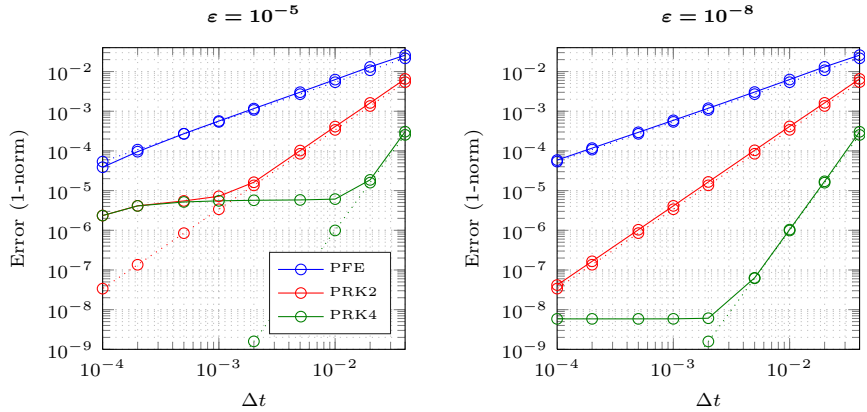

\paragraph{Numerical time order (figure \ref{fig:order_test_time_lin_adv})}
The temporal order of the projective integration methods is demonstrated in a slightly different manner as outlined above. Now, we fix the grid spacing $\dx = 0.05$ and vary the outer time step $\Dt$ as
\begin{equation} \label{eq:order_test_Dts} 
	\Dt = [0.04; 0.02; 0.01; 0.005; 0.002; 0.001; 0.0005; 0.0002; 0.0001].
\end{equation}
We will calculate the error on time $T=0.04$. The other simulation parameters remain the same as above. The error is now calculated by taking the 1-norm of the difference between the numerical solution and the analytical solution of the (linear) semi-discretized system \eqref{eq:semidiscrete}. By doing so we take into account the discretization error in space such that we only look at the error in time. The simulations are run for PFE, PRK2 and PRK4 with FE as inner integrator and upwind differences of order 3 in space. The results can be seen in figure \ref{fig:order_test_time_lin_adv} in which we also look at the influence of the value of $\epsi$ by choosing $\epsi=10^{-5}$ (left hand side plot) and $\epsi=10^{-8}$ (right hand side plot). It is clearly indicated on the plots that for small values of $\Dt$ the error curves level off towards the value of the dominant term of $O(\epsi)$ in expression \eqref{eq:local_trunc_error} since the other terms in \eqref{eq:local_trunc_error} are negligible. For $\epsi=10^{-5}$ the numerical orders are 1.07 and 2.00 for PFE and PRK2, respectively. In this case there were too few meaningful points to reliably estimate the numerical order of the PRK4 method. For $\epsi=10^{-8}$ the numerical orders are 1.02, 2.00 and 4.07, corresponding to the expected time order of the different methods.

\subsection{Burgers' equation in 1D} \label{subsec:burgers}
As a second example, we consider the inviscid Burgers' equation in one spatial dimension, 
\begin{equation} \label{eq:burgers_eq}
	\partial_t u + \partial_x\left(\frac{u^2}{2}\right) = 0.
\end{equation}
We compute the solution for $t \in [0,1]$ and $x \in [0,2]$. We impose periodic boundary conditions and consider three different initial conditions: a Gaussian pulse $u_1(x)$, a sinc wave packet $u_2(x)$ and a sine wave $u_3(x)$ given by,
\begin{equation} \label{eq:burgers_eq_init_cond} 
	u_1(x) = \exp(-25(x-1)^2), \quad\quad u_2(x) = \text{sinc}(5(x - 1)), \quad\quad u_3(x) = \sin(\pi x).
\end{equation}

For the relaxation method, we use the kinetic equation \eqref{eq:kin_eq_1d} together with the more realistic form of the Maxwellian given in \eqref{eq:maxwellian_realistic_1d} with $F(u)=u^2/2$. (Note that, compared to the previous example with linear advection, this only requires a change in one line of the code.) We discretize the velocity space using $J=2$ velocities which are obtained as the nodes of Gauss-Hermite quadrature with $\sigma^2 = 1$ in equation \eqref{eq:V_measure_gaussian}. The inner integrator is a space-time discretization of equation \eqref{eq:kin_eq_1d}, in which we again choose the standard upwind spatial discretizations of order $1$, $2$ and $3$ with grid spacing $\dx$ (that will vary throughout the experiments), combined with a forward Euler time discretization with $\dt=\epsi$ and $\epsi=10^{-8}$.  We also consider the third order ENO scheme. The projective integration method uses $K=2$ inner steps, and an outer time step of size $\Dt$ (that will also vary). 

We first perform a numerical simulation using a third order ENO spatial discretization and a fourth order projective Runge-Kutta method (PRK4) with $\dx = 10^{-2}$ and $\Dt = 5\cdot 10^{-3}$. The results are shown in figure \ref{fig:burgers_PRK4_FE_ENO3}. We clearly see that the discontinuities are nicely captured without the appearance of spurious oscillations. 

\begin{figure}[t]
	\begin{center}
		\figname{burgers_eq_1d}
%
%
\newcommand{\figWidth}{0.25\textwidth} 
\newcommand{\figHeight}{4cm} 
\newcommand{\figSpacingRight}{1.5cm} 
\begin{tikzpicture} 

\begin{axis}[%
width=0,
height=0,
scale only axis,
clip=false,
xmin=0,
xmax=1,
ymin=0,
ymax=1,
hide axis
]
\node[align=center, text=black] at (3*\figWidth/2+\figSpacingRight,\figHeight+0.9cm) {Evolution of Burgers' equation for several waves (PRK4 + FE + ENO3)};
\end{axis}

\begin{axis}[%
width=\figWidth,
height=\figHeight,
at={(0,0.0)},
scale only axis,
xmin=0,
xmax=2,
xlabel={$x$},
ymin=-0.001,
ymax=1,
ytick = {0, 0.2, ..., 1},
ylabel={$u_1$},
axis background/.style={fill=white},
title style={font=\scriptsize\bfseries},
title={pulse},
axis x line*=bottom,
axis y line*=left
]
\addplot [color=blue,solid,line width=1.0pt,forget plot]
  table[]{tikz/data/burgers_1d-13.tsv};
\addplot [color=blue,dashed,forget plot]
  table[]{tikz/data/burgers_1d-14.tsv};
\addplot [color=blue,dashed,forget plot]
  table[]{tikz/data/burgers_1d-15.tsv};
\addplot [color=blue,dashed,forget plot]
  table[]{tikz/data/burgers_1d-16.tsv};
\addplot [color=red,solid,line width=1.0pt,forget plot]
  table[]{tikz/data/burgers_1d-17.tsv};
\end{axis}

\begin{axis}[%
width=\figWidth,
height=\figHeight,
at={(\figWidth+\figSpacingRight,0)},
scale only axis,
xmin=0,
xmax=2,
xlabel={$x$},
ymin=-0.24,
ymax=1,
ytick = {-0.2, 0, ..., 1},
ylabel={$u_2$},
axis background/.style={fill=white},
title style={font=\scriptsize\bfseries},
title={sinc wave},
axis x line*=bottom,
axis y line*=left
]
\addplot [color=blue,solid,line width=1.0pt,forget plot]
  table[]{tikz/data/burgers_1d-1.tsv};
\addplot [color=blue,dashed,forget plot]
  table[]{tikz/data/burgers_1d-2.tsv};
\addplot [color=blue,dashed,forget plot]
  table[]{tikz/data/burgers_1d-3.tsv};
\addplot [color=blue,dashed,forget plot]
  table[]{tikz/data/burgers_1d-4.tsv};
\addplot [color=red,solid,line width=1.0pt,forget plot]
  table[]{tikz/data/burgers_1d-5.tsv};
\addplot [color=black,solid,forget plot]
  table[]{tikz/data/burgers_1d-6.tsv};
\end{axis}

\begin{axis}[%
width=\figWidth,
height=\figHeight,
at={(2*\figWidth+2*\figSpacingRight,0)},
scale only axis,
xmin=0,
xmax=2,
xlabel={$x$},
ymin=-1.001,
ymax=1,
ytick = {-1, -0.8, ..., 1},
yticklabel style={/pgf/number format/.cd,fixed},
ylabel={$u_3$},
axis background/.style={fill=white},
title style={font=\scriptsize\bfseries},
title={sine wave},
axis x line*=bottom,
axis y line*=left
]
\addplot [color=blue,solid,line width=1.0pt,forget plot]
  table[]{tikz/data/burgers_1d-7.tsv};
\addplot [color=blue,dashed,forget plot]
  table[]{tikz/data/burgers_1d-8.tsv};
\addplot [color=blue,dashed,forget plot]
  table[]{tikz/data/burgers_1d-9.tsv};
\addplot [color=blue,dashed,forget plot]
  table[]{tikz/data/burgers_1d-10.tsv};
\addplot [color=red,solid,line width=1.0pt,forget plot]
  table[]{tikz/data/burgers_1d-11.tsv};
\addplot [color=black,solid,forget plot]
  table[]{tikz/data/burgers_1d-12.tsv};
\end{axis}
\end{tikzpicture}%
	\end{center}
	\vspace{-0.4cm}\caption{\label{fig:burgers_PRK4_FE_ENO3} Evolution of the numerical solution of Burgers' equation obtained with PRK4 and FE as inner integrator and a third order ENO scheme \added{using $\dx=10^{-2}$}. }
\end{figure}
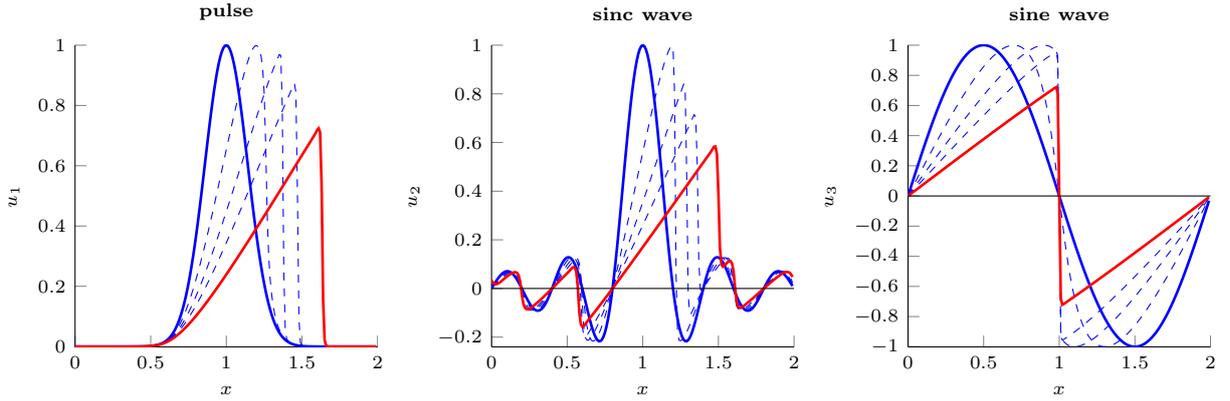

Next, we also investigate the temporal order of the methods. Since the analytical solution of Burgers' equation is not available explicitly, the error is computed with respect to a reference solution that is obtained using a high-order simulation of \eqref{eq:burgers_eq} with the PRK4 method with FE as inner integrator and upwind differences of order 3 in space with grid spacing $\dx = 0.05$ and time step $\dt = 10^{-8}$. The outer time step is chosen to be $\Dt=10^{-6}$. Then, we vary $\Dt$ as given by \eqref{eq:order_test_Dts} and we again examine the influence of the value of $\epsi$ by choosing $\epsi=10^{-5}$ (left hand side plot) and $\epsi=10^{-8}$ (right hand side plot) in figure \ref{fig:order_test_time_burgers}. It is observed on the plots that for small values of $\Dt$ the error curves level off towards the value of the dominant term of $O(\epsi)$ in expression \eqref{eq:local_trunc_error} since the other terms in \eqref{eq:local_trunc_error} are negligible. For $\epsi=10^{-5}$ the numerical orders are 1.07 and 2.00 whereas for $\epsi=10^{-8}$ the numerical orders are 1.00, 2.00 and 4.18 which are in agreement with the expected time order of the different methods. For the temporal order test the initial solution consists of a Gauss curve centered around the middle of the domain $x\in[0,2]$ (i.e. function $u_1(x)$ in \eqref{eq:burgers_eq_init_cond}) and the error is calculated at $t=0.04$. The results are comparable to the linear advection case, see figure \ref{fig:order_test_time_burgers}.

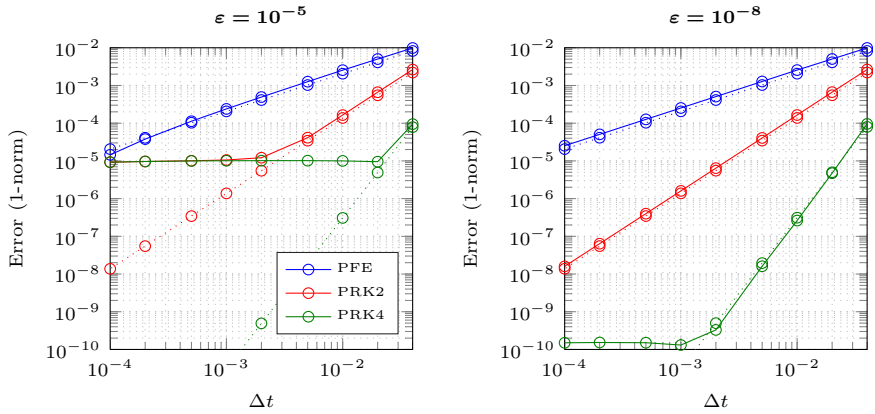
\begin{figure}[t]
	\begin{center}
		\figname{ordertest_time_burgers_eq}
%
%
\definecolor{mycolor1}{rgb}{0.00000,0.49804,0.00000}%
\newcommand{\figWidth}{0.25\textwidth} 
\newcommand{\figHeight}{4cm} 
\newcommand{\figSpacingRight}{2cm} 

\begin{tikzpicture}
\begin{axis}[%
width=0,
height=0,
scale only axis,
clip=false,
xmin=0,
xmax=1,
ymin=0,
ymax=1,
hide axis
]
\node[align=center, text=black] at (\figWidth+\figSpacingRight/2,\figHeight+0.9cm) {Order test (time) for Burgers' equation};
\end{axis}

\begin{axis}[%
width=\figWidth,
height=\figHeight,
at={(0,0.0)},
scale only axis,
xmode=log,
xmin=0.0001,
xmax=0.04,
xminorticks=true,
xlabel={$\Dt$},
xmajorgrids,
xminorgrids,
ymode=log,
ymin=1e-10,
ymax=0.01,
yminorticks=true,
ytick = {1e-10, 1e-9, 1e-8, 1e-7, 1e-6, 1e-5, 1e-4, 1e-3, 1e-2, 1e-1},
ylabel={Error (1-norm)},
ymajorgrids,
yminorgrids,
grid style={dotted},
axis background/.style={fill=white},
title style={font=\scriptsize},
title={$\boldsymbol{\epsi = 10^{-5}}$},
legend style={at={(0.55,0.05)},anchor=south west,legend cell align=left,align=left,draw=white!15!black}
]
\addplot [color=blue,solid,mark=o,mark options={solid}]
  table[]{tikz/data/ordertest_time_burgers-1.tsv};
\addlegendentry{PFE};
\addplot [color=blue,dotted,mark=o,mark options={solid},forget plot]
  table[]{tikz/data/ordertest_time_burgers-2.tsv};
\addplot [color=red,solid,mark=o,mark options={solid}]
  table[]{tikz/data/ordertest_time_burgers-3.tsv};
\addlegendentry{PRK2};
\addplot [color=red,dotted,mark=o,mark options={solid},forget plot]
  table[]{tikz/data/ordertest_time_burgers-4.tsv};
\addplot [color=mycolor1,solid,mark=o,mark options={solid}]
  table[]{tikz/data/ordertest_time_burgers-5.tsv};
\addlegendentry{PRK4};
\addplot [color=mycolor1,dotted,mark=o,mark options={solid},forget plot]
  table[]{tikz/data/ordertest_time_burgers-6.tsv};
\end{axis}

\begin{axis}[%
width=\figWidth,
height=\figHeight,
at={(\figWidth+\figSpacingRight,0)},
scale only axis,
xmode=log,
xmin=0.0001,
xmax=0.04,
xminorticks=true,
xlabel={$\Dt$},
xmajorgrids,
xminorgrids,
ymode=log,
ymin=1e-10,
ymax=0.01,
yminorticks=true,
ytick = {1e-10, 1e-9, 1e-8, 1e-7, 1e-6, 1e-5, 1e-4, 1e-3, 1e-2, 1e-1},
ylabel={Error (1-norm)},
ymajorgrids,
yminorgrids,
grid style={dotted},
axis background/.style={fill=white},
title style={font=\scriptsize},
title={$\boldsymbol{\epsi = 10^{-8}}$}
]
\addplot [color=blue,solid,mark=o,mark options={solid},forget plot]
  table[]{tikz/data/ordertest_time_burgers-7.tsv};
\addplot [color=blue,dotted,mark=o,mark options={solid},forget plot]
  table[]{tikz/data/ordertest_time_burgers-8.tsv};
\addplot [color=red,solid,mark=o,mark options={solid},forget plot]
  table[]{tikz/data/ordertest_time_burgers-9.tsv};
\addplot [color=red,dotted,mark=o,mark options={solid},forget plot]
  table[]{tikz/data/ordertest_time_burgers-10.tsv};
\addplot [color=mycolor1,solid,mark=o,mark options={solid},forget plot]
  table[]{tikz/data/ordertest_time_burgers-11.tsv};
\addplot [color=mycolor1,dotted,mark=o,mark options={solid},forget plot]
  table[]{tikz/data/ordertest_time_burgers-12.tsv};
\end{axis}
\end{tikzpicture}%
	\end{center}
	\vspace{-0.4cm}\caption{\label{fig:order_test_time_burgers} Temporal order test for Burgers' equation for the different projective integration methods under study: PFE, PRK2 and PRK4 with FE as inner integrator and upwind differences of order 3 in space comparing results for $\epsi=10^{-5}$ (left plot) and $\epsi=10^{-8}$ (right plot). The error is computed using the 1-norm. On each plot the solid lines represent the calculated error whereas the dotted line shows the expected error. }
\end{figure}

\subsection{Sod's shock test in 1D} \label{subsec:sod}
Sod's shock test is an important numerical test to check how a numerical method captures shock waves \cite{sod1978survey}. The test involves the Euler equations in one spatial dimension for mass, momentum and energy, $\uSys=(\rho,\rho\vMacroOneD,E)$,
\begin{equation} \label{eq:Euler_eq}
	\begin{dcases}
		\partial_t \rho + \partial_x (\rho\vMacroOneD) = 0, \\
		\partial_t (\rho\vMacroOneD) + \partial_x (\rho\vMacroOneD^2 + P) = 0, \\
		\partial_t E + \partial_x (\added{(E + P)\vMacroOneD}) = 0,
	\end{dcases}
\end{equation}
in which the pressure $P$ is determined via the equation of state:
\begin{equation} \label{eq:euler_eq_state_1d_again}
	P = (\gamma - 1)\left(E - \frac{1}{2}\rho\vMacroOneD^2\right), 
\end{equation}
in which the constant $\gamma$ equals $7/5$ in case of a diatomic perfect gas \cite{sod1978survey}.

As an initial condition, Sod's shock test imposes
\begin{equation} \label{eq:Euler_eq_IC}
	\begin{aligned}
		\rho(x,0) &= \left\{ \begin{array} {l@{\qquad \quad}l} 1 & x \le L/2\\ 0.125 & x > L/2 \end{array} \right.&, \\
		P(x,0) &=  \left\{ \begin{array} {l@{\qquad \quad}l} 1 & x \le L/2\\ 0.1 & x > L/2 \end{array} \right.&, \\
		\vMacroOneD(x,0) &= 0&.
	\end{aligned}
\end{equation}
Additionally, we impose outflow boundary conditions. We perform the simulation over the spatial domain $x \in [0,1]$ and $t \in [0,0.22]$. The particular choice of the time interval allows for a clear visualization of the three different characteristic waves (see below) and an easy comparison with the available literature.

For the relaxation method, we use the kinetic equation \eqref{eq:kin_eq_1d}, in which we discretize the velocity space with $J=2$ velocities corresponding to the nodes of Gauss-Hermite quadrature. The Maxwellian is chosen as
\begin{equation} \label{eq:sod_maxwellian} 
	\Mav(\uSys^\epsi) = \left( \begin{array}{c} \rho + v(\rho\vMacroOneD) \\ \rho\vMacroOneD + v(\rho\vMacroOneD^2 + P) \\ E + v(E + P)\vMacroOneD \end{array} \right).
\end{equation}
The inner integrator uses a third order spatial ENO discretization with $\dx = 5\cdot 10^{-3}$ and a forward Euler time discretization with $\dt = \epsi = 10^{-8}$.  As the outer method, we choose the fourth order projective Runge-Kutta method (PRK4), using $K=2$ inner steps and an outer step of size $\Dt = 0.5\dx$.

The result is illustrated in figure \ref{fig:sod_PRK4_ENO3}, where we plot density, velocity, energy and pressure at time $t=0.22$, along with the analytical solution, calculated with an exact Riemann solver, for comparison.

\begin{figure}[t]
	\begin{center}
		\figname{euler_eq_1d}
%
%
\newcommand{\figWidth}{0.3\textwidth} 
\newcommand{\figHeight}{3.9cm} 
\newcommand{\figSpacingRight}{1.5cm} 
\newcommand{\figSpacingTop}{1.5cm} 
\begin{tikzpicture} 

\begin{axis}[%
width=0,
height=0,
scale only axis,
clip=false,
xmin=0,
xmax=1,
ymin=0,
ymax=1,
hide axis,
]
\node[align=center, text=black]
at (\figWidth+\figSpacingRight/2,2*\figHeight+\figSpacingTop+0.9cm) {Evolution of Sod's shock test (PRK4 + FE + ENO3)};
\end{axis}

\begin{axis}[%
width=\figWidth,
height=\figHeight,
at={(0,\figHeight+\figSpacingTop)},
scale only axis,
xmin=0,
xmax=1,
xlabel={$x$},
ymin=0,
ymax=1.1,
ytick = {0, 0.2, ..., 1},
ylabel={$\rho$},
axis background/.style={fill=white},
title style={font=\scriptsize\bfseries},
title={Density},
axis x line*=bottom,
axis y line*=left,
clip mode=individual 
]
\addplot [color=blue,solid,mark=o,mark options={solid},forget plot]
  table[]{tikz/data/euler_1d-1.tsv};
\addplot [color=red,solid,line width=1.0pt,forget plot]
  table[]{tikz/data/euler_1d-2.tsv};
\end{axis}

\begin{axis}[%
width=\figWidth,
height=\figHeight,
at={(\figWidth+\figSpacingRight,\figHeight+\figSpacingTop)},
scale only axis,
xmin=0,
xmax=1,
xlabel={$x$},
ymin=-0.05,
ymax=1.05,
ytick = {0, 0.2, ..., 1},
ylabel={$\vMacroOneD$},
axis background/.style={fill=white},
title style={font=\scriptsize\bfseries},
title={Velocity},
axis x line*=bottom,
axis y line*=left,
clip mode=individual 
]
\addplot [color=blue,solid,mark=o,mark options={solid},forget plot]
  table[]{tikz/data/euler_1d-3.tsv};
\addplot [color=red,solid,line width=1.0pt,forget plot]
  table[]{tikz/data/euler_1d-4.tsv};
\end{axis}

\begin{axis}[%
width=\figWidth,
height=\figHeight,
at={(0,0)},
scale only axis,
xmin=0,
xmax=1,
xlabel={$x$},
ymin=0,
ymax=2.75,
ytick = {0, 0.5, ..., 2.5},
ylabel={$E$},
axis background/.style={fill=white},
title style={font=\scriptsize\bfseries},
title={Energy},
axis x line*=bottom,
axis y line*=left,
clip mode=individual 
]
\addplot [color=blue,solid,mark=o,mark options={solid},forget plot]
  table[]{tikz/data/euler_1d-5.tsv};
\addplot [color=red,solid,line width=1.0pt,forget plot]
  table[]{tikz/data/euler_1d-6.tsv};
\end{axis}

\begin{axis}[%
width=\figWidth,
height=\figHeight,
at={(\figWidth+\figSpacingRight,0)},
scale only axis,
xmin=0,
xmax=1,
xlabel={$x$},
ymin=0,
ymax=1.1,
ytick = {0, 0.2, ..., 1},
ylabel={$P$},
axis background/.style={fill=white},
title style={font=\scriptsize\bfseries},
title={Pressure},
axis x line*=bottom,
axis y line*=left,
clip mode=individual 
]
\addplot [color=blue,solid,mark=o,mark options={solid},forget plot]
  table[]{tikz/data/euler_1d-7.tsv};
\addplot [color=red,solid,line width=1.0pt,forget plot]
  table[]{tikz/data/euler_1d-8.tsv};
\end{axis}
\end{tikzpicture}%
	\end{center}
	\vspace{-0.4cm}\caption{\label{fig:sod_PRK4_ENO3} Fluid properties in Sod's shock test at $t = 0.22$ obtained with PRK4 and FE as inner integrator and a third order ENO discretization in space \added{using $\dx = 5 \cdot 10^{-3}$}. }
\end{figure}
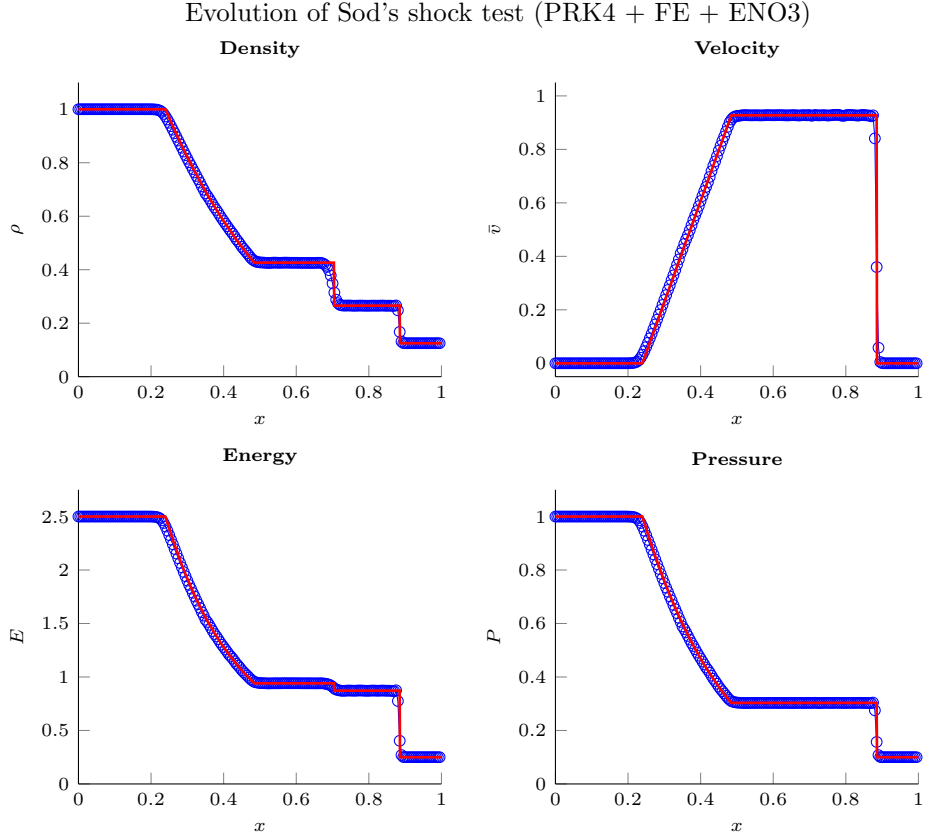

We clearly see the development of three characteristic waves. The first wave corresponds to a rarefaction wave propagating to the left since initially the pressure and density on the left half side of the shock tube are higher than on the right side. Secondly, a contact discontinuity is observed. This right propagating wave corresponds to the initial discontinuity of \eqref{eq:Euler_eq_IC}. Finally, a shock wave propagating to the right has also appeared. It appears that only pressure and velocity remain continuous over the contact discontinuity and exhibit a very flat state. Furthermore, it can be observed that all the quantities are discontinuous over the shock wave. The projective integration method captures all these phenomena, without developing undesired oscillations and without smoothing out the discontinuities too much. 

\subsection{Linear advection in 2D} \label{subsec:lin_2d}
Let us now turn to problems in two spatial dimensions. We again start with the linear advection equation, which, in dimension 2, reads
\begin{equation} \label{eq:2D_lin_adv}
	\partial_t u + a\partial_x u + b\partial_y u = 0,
\end{equation}
in which $a,b \in \R$ are the constant advection speeds along the $x$- and $y$-direction, respectively. The macroscopic unknown function $u(x,y,t)$ denotes the two-dimensional density of particles. In the simulations we integrate over $t \in [0,1]$ and $(x,y) \in [0,1]^2$, and set $a = b = 1$. We impose periodic boundary conditions and start from a Gaussian pulse centered in the middle of the domain:
\begin{equation} \label{eq:2D_lin_adv_IC} 
	u(x,y,0) = \exp\left(-50(x-0.5)^2\right)\exp\left(-50(y - 0.5)^2\right).
\end{equation}

As described in section \ref{sec:2D} we will now solve the two-dimensional kinetic equation \eqref{eq:kin_eq_2d} with the Maxwellian given by \eqref{eq:maxwellian_2d}. The velocity discretization is now determined by the orthogonal velocity method as follows: we fix $R = S = 1$ and calculate $\vmax$ by choosing the (integer) lower bound from expression \eqref{eq:vmax_2d},
\[ \vmax = \left\lceil \sqrt{\frac{12R^2(a^2+b^2)}{(R+1)(2R+1)}}\; \right\rceil. \] 
The inner integrator is a space-time discretization of the kinetic equation \eqref{eq:kin_eq_2d}, in which we take the ENO scheme of order 1, 2 and 3 in space with $\dx = \dy = 0.04$ and the forward Euler scheme in time with $\dt = \epsi = 10^{-8}$. The outer integrator is the fourth order projective Runge-Kutta (PRK4) method, using $K=2$ and $\Dt = 0.3\dx$. 

We compare the obtained numerical results for increasing order in space from $1$ to $3$. This is shown in figure \ref{fig:2D_linadv_PRK4_FE_ENO}.

\begin{figure}[t]
	\begin{center}
		\figname{lin_adv_eq_2d}
		\input{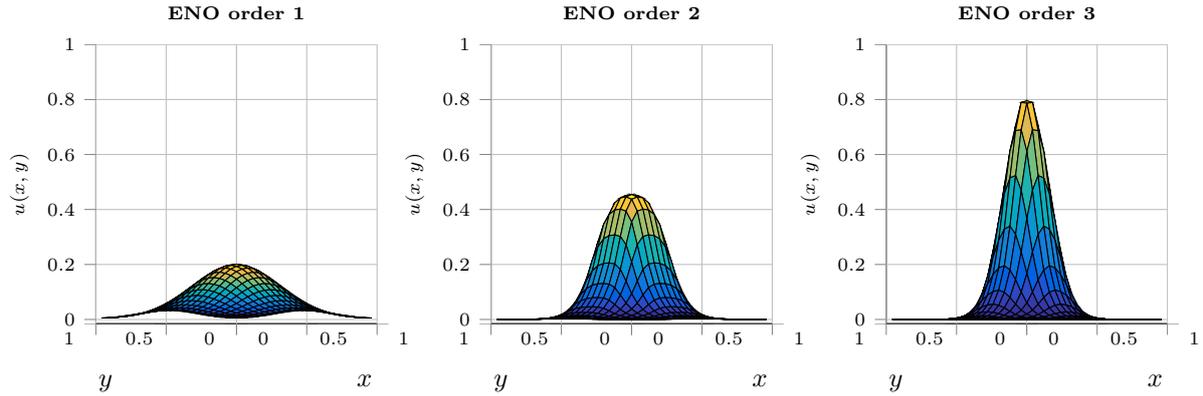}
	\end{center}
	\vspace{-0.4cm}\caption{\label{fig:2D_linadv_PRK4_FE_ENO} Comparison of the numerical solution for linear advection in 2D at $t = 1$ obtained by PRK4 and FE as inner integrator for different orders in space of the ENO scheme \added{with $\dx = \dy = 0.04$}.}
\end{figure}

\subsection{Dam-break problem in 2D}\label{subsec:shawa_2d}
In dam-break problems, one is interested in the evolution of two (or more) regions of water that are separated by a dam, which then is suddenly removed. Such problems are typically modeled by shallow water equations and can be understood as the equivalent of shock tube applications in gas dynamics (see section \ref{subsec:sod}) \cite{LeVeque2002}. In two space dimensions, the shallow water equations take the following form:
\begin{equation}\label{2D_shawa_eq}
	\begin{dcases}
		\partial_t h + \partial_x (h\vMacroOneD_x) + \partial_y (h\vMacroOneD_y) = 0 \\
		\partial_t (h\vMacroOneD_x) + \partial_x (h\vMacroOneD_x^2 + P) + \partial_y (h\vMacroOneD_x \vMacroOneD_y) = 0 \\
		\partial_t (h\vMacroOneD_y) + \partial_x (h\vMacroOneD_x \vMacroOneD_y) + \partial_y (h\vMacroOneD_y^2 + P) = 0, \\
	\end{dcases}
\end{equation}
in which $h$ is the depth of the water, ${\vMacroMultiD = \left(\vMacroOneD_x,\vMacroOneD_y\right)}$ is the (macroscopic) velocity vector, and $P$ is the pressure. All these unknown functions depend on $x, y$ and $t$. The system in \eqref{2D_shawa_eq} contains three partial differential equations for four unknown functions: $h, \vMacroOneD_x, \vMacroOneD_y$ and $P$. Therefore we close the system by the following (hydrostatic) equation of state:
\begin{equation} \label{eq:2D_shawa_eq_of_state}
	P = \frac{1}{2}gh^2.
\end{equation}
The shallow water system is considered over the spatial domain $(x,y)\in[-2.5,2.5]^2$ and $t\in[0,1.5]$.

The initial solution consists of a cylindrical basin of water surrounded by a dam (see \cite{LeVeque2002}), given by:
\begin{equation} \label{eq:2D_shawa_eq_IC}
	\begin{aligned}
		h(x,y,0) &= \left\{ \begin{array} {l@{\qquad \quad}l} 2 & x^2+y^2 \le 0.5 \\ 1 & \text{otherwise} \end{array} \right.& \\
		\vMacroMultiD(x,y,0) &= 0.&
	\end{aligned}
\end{equation}
Furthermore, we impose outflow boundary conditions.

For the relaxation method, we use the kinetic equation \eqref{eq:kin_eq_2d} together with the Maxwellian given by \eqref{eq:maxwellian_2d}. The velocity discretization is determined by the orthogonal velocity method as follows: we fix $R=S=1$ and calculate $\vmax$ as follows,
\begin{equation} \label{eq:sha_wa_2d_vmax}
	\vmax = \left\lceil \sqrt{\frac{24R^2}{(R+1)(2R+1)}}\; \right\rceil.
\end{equation}
In the projective integration framework, the inner integrator is a space-time discretization of the kinetic equation \eqref{eq:kin_eq_2d}, in which we take the upwind scheme of order 3 in space with $\dx = \dy = 0.01$ and the forward Euler scheme in time with $\dt = \epsi = 10^{-8}$. The outer integrator is the fourth order projective Runge-Kutta (PRK4) method, using $K=2$ and $\Dt = 0.3\dx$. The contour plots of the water depth $h$ can be seen in figure \ref{fig:2D_shawa_PRK4_u3}. As in \cite{LeVeque2002}, we find that at $t = 1.5$ the depth of the water near the origin stabilizes around $h \approx 0.96$.

\begin{figure}
	\begin{center}
		\includegraphics[scale=0.60]{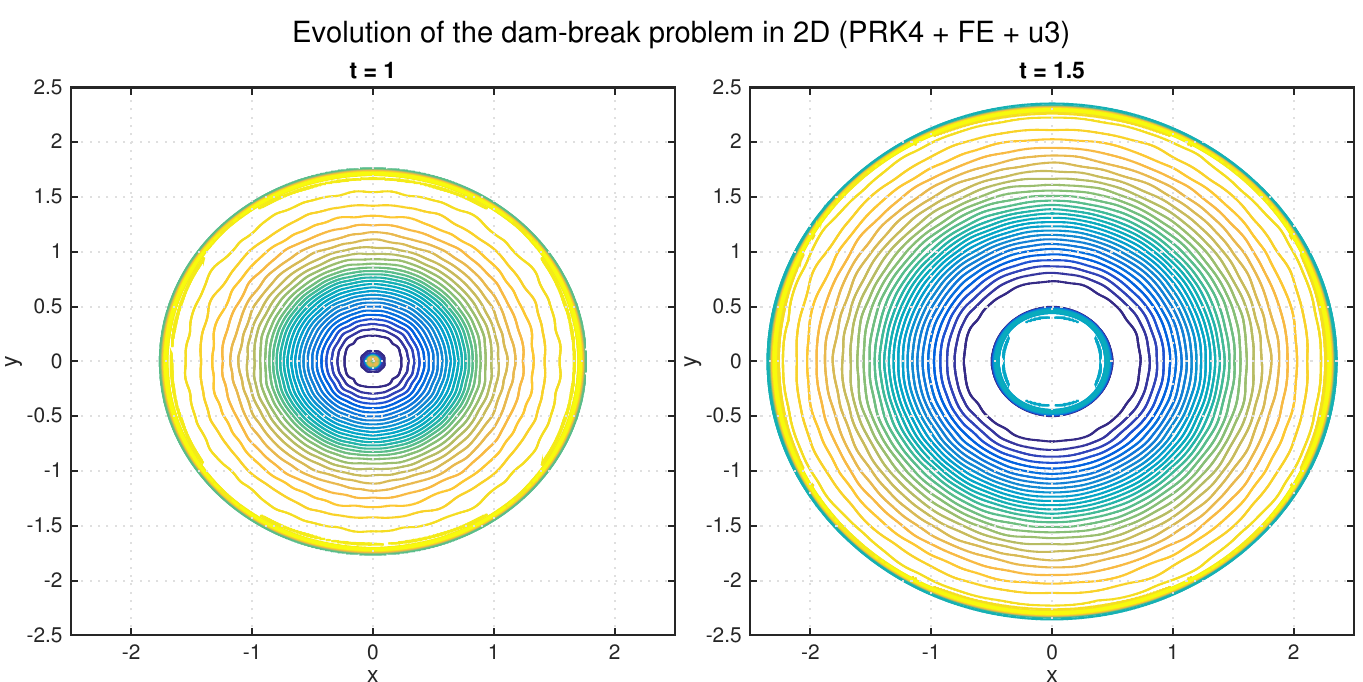}
	\end{center}
	\vspace{-0.4cm}\caption{\label{fig:2D_shawa_PRK4_u3} Contour plots depicting the water depth $h(x,y,t)$ on $t=1$ (left) and $t=1.5$ (right) in the cylindrical dam-break problem obtained with PRK4 and FE as inner integrator and the third order upwind scheme \added{using $\dx = \dy = 10^{-2}$}. }
\end{figure}

\subsection{Euler equations in 2D} \label{subsec:euler_2d}
When extending the Euler equations from 1D to 2D we end up with a system of four equations since there is conservation of momentum along both the $x$- and $y$-direction. This leads to the following system of equations:
\begin{equation} \label{eq:2D_Euler_eq}
	\begin{dcases}
		\partial_t \rho + \partial_x (\rho\vMacroOneD_x) + \partial_y (\rho\vMacroOneD_y) = 0 \\
		\partial_t (\rho\vMacroOneD_x) + \partial_x (\rho\vMacroOneD_x^2 + P) + \partial_y (\rho\vMacroOneD_x \vMacroOneD_y) = 0 \\
		\partial_t (\rho\vMacroOneD_y) + \partial_x (\rho\vMacroOneD_x \vMacroOneD_y) + \partial_y (\rho\vMacroOneD_y^2 + P) = 0 \\
		\partial_t E + \partial_x ((E + P)\vMacroOneD_x) + \partial_y ((E + P)\vMacroOneD_y) = 0.
	\end{dcases}
\end{equation}
In system \eqref{eq:2D_Euler_eq} the unknown functions $\rho$, ${\vMacroMultiD = \left(\vMacroOneD_x,\vMacroOneD_y\right)}$, $P$ and $E$ all depend on $x, y$ and $t$. Similarly to the one-dimensional situation, system \eqref{eq:2D_Euler_eq} now consists of four partial differential equations for five unknown functions: $\rho, \vMacroOneD_x, \vMacroOneD_y, P$ and $E$. Therefore we close the system by the following equation of state:
\begin{equation} \label{eq:euler_eq_state_2d}
	P = (\gamma - 1)\left(E - \frac{1}{2}\rho\norm{\boldsymbol{\bar v}}^2\right),
\end{equation} 
with $\gamma = 7/5$. We consider the Euler system over the spatial domain $(x,y) \in [-0.5,0.5]^2$ and $t \in [0,0.18]$.

The extension of the initial discontinuous configuration \eqref{eq:Euler_eq_IC} of Sod's shock test in 1D to 2D is given by (see \cite{Aregba-Driollet2000}):
\begin{equation} \label{eq:2D_Euler_eq_IC}
	\begin{aligned}
		\rho(x,y,0) &= \left\{ \begin{array} {l@{\qquad \quad}l} 0.1 & xy \le 0 \\ 1 & \text{otherwise} \end{array} \right.& \\
		P(x,y,0) &= \left\{ \begin{array} {l@{\qquad \quad}l} 0.1 & xy \le 0 \\ 1 & \text{otherwise} \end{array} \right.& \\
		\vMacroMultiD(x,y,0) &= 0,&
	\end{aligned}
\end{equation}
which is also called a double Sod tube. Furthermore, we impose outflow boundary conditions.

For the relaxation method and velocity discretization, we use the same setting as in section \ref{subsec:shawa_2d}.
In the projective integration framework, the inner integrator is a space-time discretization of the kinetic equation \eqref{eq:kin_eq_2d}, in which we take the ENO scheme of order 3 in space with $\dx = \dy = 0.01$ and the forward Euler scheme in time with $\dt = \epsi = 10^{-8}$. The outer integrator is the fourth order projective Runge-Kutta (PRK4) method, using $K=2$ and $\Dt = 0.3\dx$. The results can be seen in figure \ref{fig:2D_sod_PRK4_ENO3} and correspond to those in \cite{Aregba-Driollet2000}.

\begin{figure}[h!]
	\begin{center}
		\includegraphics[scale=0.60]{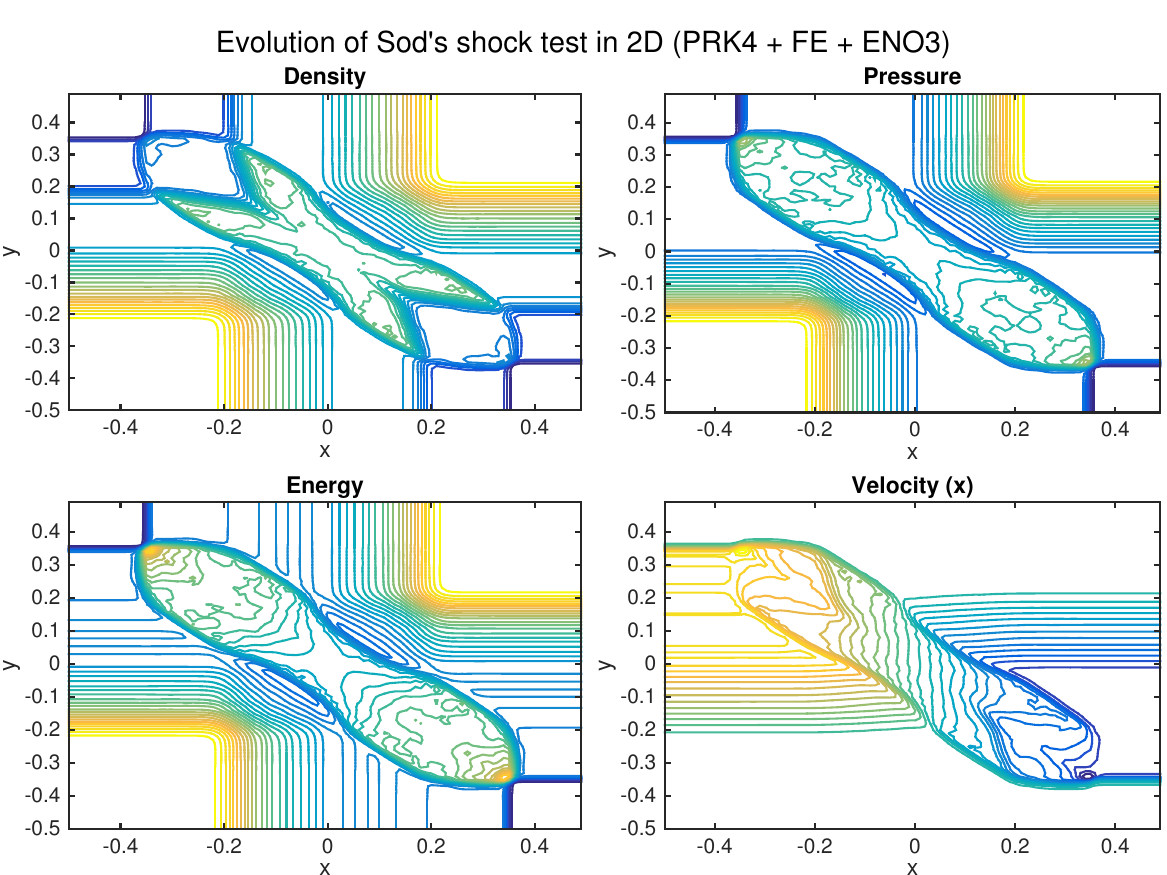}
	\end{center}
	\vspace{-0.4cm}\caption{\label{fig:2D_sod_PRK4_ENO3} Fluid properties in Sod's shock test in 2D at $t = 0.18$ obtained with PRK4 and FE as inner integrator and a third order ENO scheme \added{using $\dx = \dy = 10^{-2}$}. }
\end{figure}


\section{Conclusions} \label{sec:conclusions}
We presented a general, high-order, fully explicit, relaxation scheme for systems of nonlinear hyperbolic conservation laws in multiple dimensions, by approximating the nonlinear hyperbolic conservation law by a kinetic equation with BGK source term, which is, in turn, discretized and integrated using a projective integration method. After taking a few small (inner) steps with the direct forward Euler method, an estimate of the time derivative is used in an (outer) Runge-Kutta method of arbitrary order.

Unlike other methods based on relaxation \cite{Jin1995,Aregba-Driollet2000}, the projective integration method does not rely on a splitting technique, but only on an appropriate selection of time steps using a naive explicit discretization method.  Its main advantage is its generality and ease of use: implementing the method for a different system of hyperbolic conservation laws only requires changing the definition of the Maxwellian.

We showed that, with an appropriate choice of inner step size, the time step restriction on the outer time step is similar to the CFL condition for the hyperbolic conservation law. Moreover, the number of inner time steps is also independent of the scaling parameter. We analyzed stability and consistency, and illustrated with numerical results on a set of test problems of varying complexity.

\added{For future work, it is interesting to extend and analyze the proposed technique to quasilinear systems of conservation laws and hyperbolic equations with source terms.}

\bibliographystyle{plain} \bibliography{refs}

\end{document}